\documentclass[a4paper,10pt]{article}

\usepackage[english]{babel}
\usepackage{a4wide}
\usepackage[utf8]{inputenc}
\usepackage{amssymb,amsmath,amscd,amsfonts,amsthm,bbm,mathrsfs,enumerate}
\usepackage[shortlabels]{enumitem}
\usepackage[colorlinks=true, linkcolor=blue, citecolor=blue]{hyperref}
\usepackage{graphicx}
\usepackage{color}
\usepackage{yhmath}
\usepackage{accents}

\DeclareUnicodeCharacter{00A0}{~}

\theoremstyle{definition}
\newtheorem{theorem}{Theorem}[section]
\newtheorem*{theorem*}{Statement}
\newtheorem{lemma}[theorem]{Lemma}

\newtheorem{proposition}[theorem]{Proposition}
\newtheorem{definition}{Definition}[section]

\newtheorem*{condition*}{Condition}

\newtheorem*{assumption*}{Assumption}
\newtheorem{example}{Example}[section]
\newtheorem*{example*}{Example}

\DeclareMathOperator{\dom}{dom}
\DeclareMathOperator{\zer}{zer}
\DeclareMathOperator{\aver}{av}

\DeclareMathOperator{\graph}{gr}
\DeclareMathOperator{\prox}{prox}

\DeclareMathOperator{\support}{supp}

\DeclareMathOperator{\co}{co}

\usepackage[textwidth=2cm, textsize=footnotesize]{todonotes}  

\newcommand{\1}{\mathbbm 1}

\newcommand{\sx}{{\mathsf x}}
\newcommand{\sy}{{\mathsf y}}
\newcommand{\sz}{{\mathsf z}}

\newcommand{\sH}{{\mathsf H}}

\newcommand{\bP}{{{\mathbb P}}} 
\newcommand{\bE}{{{\mathbb E}}} 
\newcommand{\bN}{{{\mathbb N}}} 

\newcommand{\mU}{{\mathcal U}}

\newcommand{\mZ}{{\mathcal Z}}

\newcommand{\sA}{{\mathsf A}}

\newcommand{\sX}{{\mathsf X}}
\newcommand{\sG}{{\mathsf G}}

\newcommand{\Selec}{{\mathfrak S}} 

\newcommand{\mm}{{\mathfrak m}} 

\newcommand{\mcB}{{\mathscr B}}

\newcommand{\mcI}{{\mathscr I}} 
\newcommand{\mcF}{{\mathscr F}} 
\newcommand{\mcG}{{\mathscr G}}

\newcommand{\cP}{{{\mathcal P}}} 
\newcommand{\cS}{{{\mathcal S}}} 
 
\newcommand{\cG}{{{\mathcal G}}} 
\newcommand{\cM}{{{\mathcal M}}} 
 
\newcommand{\cE}{{{\mathcal E}}} 
 
\newcommand{\cL}{{{\mathcal L}}}
\newcommand{\cT}{{{\mathcal T}}}

\newcommand{\cI}{{{\mathcal I}}} 
 
\newcommand{\cC}{{{\mathcal C}}} 

\newcommand{\RN}{{{\mathbb R}^N}} 
\newcommand{\bR}{{{\mathbb R}}}

\newcommand{\ps}[1]{\langle #1 \rangle}

\newcommand{\bs}{\boldsymbol}
\newcommand{\eqdef}{:=} 
\newcommand{\rev}[1]{#1}

\begin{document}

\title{\rev{Constant Step Stochastic Approximations \\
Involving Differential Inclusions: Stability, \\
Long-Run Convergence and Applications}}  

\author{Pascal Bianchi$^{(1)}$, Walid Hachem$^{(2)}$, and Adil Salim$^{(1)}$ 
\\  
{$^{(1)}$ LTCI, T\'el\'ecom ParisTech, Universit\'e Paris-Saclay, 
  75013, Paris, France} \\ 
{$^{(2)}$ CNRS / LIGM (UMR 8049), Universit\'e Paris-Est
Marne-la-Vall\'ee} \\
\texttt{pascal.bianchi,adil.salim@telecom-paristech.fr,walid.hachem@u-pem.fr}}

\maketitle

\begin{abstract}
\rev{We consider a Markov chain $(x_n)$ whose kernel is indexed by a scaling
parameter $\gamma>0$, refered to as the step size.  The aim is to analyze the
behavior of the Markov chain in the doubly asymptotic regime where $n\to\infty$
then $\gamma\to 0$. First, under mild assumptions on the so-called drift of the
Markov chain, we show that the interpolated process converges narrowly to the
solutions of a Differential Inclusion (DI) involving an upper semicontinuous
set-valued map with closed and convex values.  Second, we provide verifiable
conditions which ensure the stability of the iterates.  Third, by putting the
above results together, we establish the long run convergence of the iterates
as $\gamma\to 0$, to the Birkhoff center of the DI.  The ergodic behavior of
the iterates is also provided.  Application examples are investigated.  We
apply our findings to 1) the problem of
nonconvex proximal stochastic optimization and 2) a fluid model of parallel queues.  
}

  \bigskip

34A60,47N10, 54C60,62L20.

\end{abstract}

\section{Introduction} 
\label{intro}

\rev{In this paper, we consider a Markov chain $(x_n,n\in \bN)$ with values in $E = \bR^N$, where $N\geq 1$ is an integer.
We assume that the probability transition kernel $P_\gamma$ is indexed by a scaling factor $\gamma$, which 
belongs to some interval $(0,\gamma_0)$.
The aim of the paper is to analyze the long term behavior of the Markov chain in the regime where $\gamma$ is small.
The map
\begin{equation}
  \label{eq:drift}
  g_\gamma(x) \eqdef \int \frac{y-x}\gamma P_\gamma(x,dy)\,,
\end{equation}
assumed well defined for all $x\in \bR^N$, is called the \emph{drift} or the \emph{mean field}. The Markov chain admits the representation
\begin{equation}
x_{n+1} = x_n + \gamma\,g_\gamma(x_n) +  \gamma\,U_{n+1}\,,\label{eq:decomp-markov-drift}
\end{equation}
where $U_{n+1}$ is a martingale increment noise \emph{i.e.,} the conditional expectation of $U_{n+1}$ given the past samples 
is equal to zero. A case of interest in the paper is given by iterative models of the form:
\begin{equation}
x_{n+1} = x_n+\gamma \,h_\gamma (\xi_{n+1},x_n)\,,\label{eq:iterative-model}
\end{equation}
where $(\xi_n, n \in \bN^*)$ is a sequence of independent and identically distributed (iid)
random variables defined on a probability space $\Xi$ with probability law 
$\mu$, and $\{ h_\gamma \}_{\gamma\in(0,\gamma_0)}$ is a family of maps
on $\Xi\times \bR^N\to\RN$. In this case, 
the drift $g_\gamma$ has the form:
\begin{equation}
g_\gamma(x) = \int h_\gamma(s,x)\,\mu(ds)\,.\label{eq:drift-integral}
\end{equation}
Our results are as follows.

\begin{enumerate}
\item {\bf Dynamical behavior.} Assume that the drift $g_\gamma$ has the form~(\ref{eq:drift-integral}).
Assume that for $\mu$-almost all $s$ and for every sequence 
$((\gamma_k,z_k)\in (0,\gamma_0)\times \RN, k\in\bN)$ converging to $(0,z)$,
$$
h_{\gamma_k}(s,z_k) \to H(s,z)
$$
where $H(s,z)$ is a subset of $\RN$ (the Euclidean distance between $h_{\gamma_k}(s,z_k)$ and the set $H(s,z)$
tends to zero as $k\to\infty$).
Denote by $\sx_\gamma(t)$ the continuous-time stochastic process obtained by a
piecewise linear interpolation of the sequence $x_n$, where
the points $x_n$ are spaced by a fixed time step $\gamma$ on the positive real axis.
As $\gamma\to 0$, and assuming that $H(s,\cdot)$ is a proper and upper semicontinuous (usc) map with closed convex values, 
we prove that $\sx_\gamma$ converges
narrowly  (in the topology of
uniform convergence on compact sets) to the set of solutions of the differential inclusion (DI)
\begin{equation}
\dot \sx(t) \in  \int H(s, \sx(t))\mu(ds)\,,\label{eq:di-intro}
\end{equation}
where for every $x\in \RN$, $\int  H(s, x)\mu(ds)$ is the \emph{selection integral} of $H(\,.\,x)$, which is defined as
the closure of the set of integrals of the form
$\int \varphi d\mu$ where  $\varphi$ is any integrable function
such that $\varphi(s) \in H(s, x)$ for $\mu$-almost all $s$.
\smallskip

\item {\bf Tightness.} As the iterates are not \emph{a priori} supposed to be in a compact subset of $\RN$, 
we investigate the issue of stability. We posit a verifiable \emph{Pakes-Has'minskii} condition on 
the Markov chain $(x_n)$. The condition ensures that the iterates are stable in the sense that the random occupation measures
$$
\Lambda_n \eqdef \frac 1{n+1}\sum_{k=0}^n\delta_{x_k} \qquad (n\in \bN)
$$
(where $\delta_a$ stands for the Dirac measure at point $a$), form a tight family of random variables
on the Polish space of probability measures equipped with the L\'evy-Prokhorov distance. 
The same criterion allows to establish
the existence of invariant measures of the kernels $P_\gamma$, and the tightness of the 
family of all invariant measures, for all $\gamma\in (0,\gamma_0)$.
As a consequence of Prokhorov's theorem, these invariant measures admit
cluster points as $\gamma\to 0$. Under a Feller assumption on the kernel $P_\gamma$,
we prove that every such cluster point is an invariant measure for 
the DI (\ref{eq:di-intro}). Here, since the
flow generated by the DI is in general
set-valued, the notion of invariant measure is borrowed from \cite{fau-rot-13}.
\smallskip

\item {\bf Long-run convergence.}
Using the above results, we investigate the behavior of the iterates in the 
asymptotic regime where $n\to\infty$ and, next, $\gamma\to 0$.
Denoting by $d(a,B)$ the distance between a point $a\in E$ and a subset $B\subset E$,
we prove that for all $\varepsilon>0$, 
\begin{equation}
\lim_{\gamma\to 0}\limsup_{n\to\infty} \frac 1{n+1}\sum_{k=0}^{n}\text{Prob}\left(d(x_k,\text{BC})>\varepsilon\right) = 0\,,\label{eq:longrun}
\end{equation}
where $\text{BC}$ is the Birkhoff center of the flow induced by the DI (\ref{eq:di-intro}),
and $\text{Prob}$ stands for the probability.
We also characterize the ergodic behavior of these iterates. Setting $\overline x_n = \frac 1{n+1}\sum_{k=0}^{n}x_k$, 
we prove that 
\begin{equation}
\lim_{\gamma\to 0}\limsup_{n\to\infty} \text{Prob}\left(d(\overline x_n,\co(L_{\aver}))>\varepsilon\right)=0\,,\label{eq:longrun-ergodic}
\end{equation}
where $\co(L_{\aver})$ is the convex hull of
the limit set of the averaged flow associated with (\ref{eq:di-intro}) (see Section~\ref{subsec-inv-mes}).
\smallskip

\item {\bf Applications.} We investigate several application scenarios. We consider the problem
of non-convex stochastic optimization, and analyze the convergence of a constant step size proximal stochastic gradient algorithm.
The latter finds application in the optimization of deep neural networks \cite{lecun1998gradient}.
We show that the interpolated process converges narrowly to a DI, which we characterize. 
We also provide sufficient conditions allowing to characterize the long-run behavior of the algorithm.
Second, we explain that our results apply to the characterization of the fluid limit
of a system of parallel queues. The model is introduced in \cite{ayesta2013scheduling,gas-gau-12}.
Whereas the narrow convergence of the interpolated process was studied in \cite{gas-gau-12}, less is known about the stability
and the long-run convergence of the iterates. We show how our results can be used to address this problem.
As a final example, we explain how our results can be used in the context of monotone operator theory, 
in order to analyze a stochastic version of the celebrated proximal point algorithm.
The algorithm consists in replacing the usual monotone operator by an iid sequence of random monotone operators. 
The algorithm has been studied in \cite{bia-16,bia-hac-16} in the context of decreasing step size. 
Our analysis provide the tools to characterize its behavior in a constant step regime.
\end{enumerate}

\paragraph*{Paper organization.} In Section~\ref{sec:examples}, we introduce
the application examples. In Section~\ref{sec:literature}, we briefly 
discuss the literature.}
Section~\ref{sec-bgrd} is devoted to the
mathematical background and to the notations. The main results are given in
Section~\ref{sec-main}. The tightness of the interpolated process as well as
its narrow convergence towards the solution set of the DI
(Th.~\ref{th:SA=wAPT}) are proven in Section~\ref{sec-prf-narrow}.  Turning to
the Markov chain characterization, Prop.~\ref{prop:cluster}, who explores the
relations between the cluster points of the Markov chains invariant measures
and the invariant measures of the flow induced by the DI, is proven in
Section~\ref{sec-prf-cluster}.  A general result describing the asymptotic
behavior of a functional of the iterates with a prescribed growth is provided
by Th.~\ref{the:CV}, and proven in Section~\ref{sec-prf-CV}.  Finally, in
Section~\ref{sec-prf-asymptotics}, we show how the results pertaining to the
ergodic convergence and to the convergence of the iterates (Th.~\ref{cvg-CVSI}
and~\ref{cvg-XY} respectively) can be deduced from Th.~\ref{the:CV}.  
\rev{Finally, Section~\ref{sec:applis} is devoted to the application examples. 
We prove that our hypotheses are satisfied.
}

\rev{\section{Examples}
\label{sec:examples}

\begin{example} \label{ex:optim}
{\sl Non-convex stochastic optimization.}
  Consider the problem
\begin{equation}
\text{minimize } \bE_\xi(\ell(\xi,x)) + r(x)\text{ w.r.t }x\in \bR^N\,,\label{eq:pb-nonCVX}
\end{equation}
where $\ell(\xi,\,.\,)$ is a (possibly non-convex) differentiable function on $\bR^N\to \bR$ indexed by a random variable (r.v.) $\xi$,
$\bE_\xi$ represents the expectation w.r.t. $\xi$, and $r:\bR^N\to \bR$ is a convex function.
The problem typically arises in deep neural networks \cite{yoon2017combined,scardapane2017group}. In the latter case, $x$ represents the collection of weights of
the network, $\xi$ represents a random training example of the database, and $\ell(\xi,x)$ is a risk function which quantifies
the inadequacy between the sample response and the network response. Here, $r(x)$ is a regularization term which prevents
the occurence of undesired solutions. A typical regularizer used in machine learning is the $\ell_1$-norm $\|x\|_1$ that promotes sparsity or generalizations like $\|D x\|_1$, where $D$ is a matrix, that promote structured sparsity.
A popular algorithm used to find an approximate solution to Problem~(\ref{eq:pb-nonCVX}) is the proximal stochastic gradient
algorithm, which reads
\begin{equation}
  x_{n+1} = \prox_{\gamma r}(x_n - \gamma \nabla \ell(\xi_{n+1},x_n))\,,\label{eq:prox-gradient}
\end{equation}
where $(\xi_n,n\in \bN^*)$ are i.i.d. copies of the r.v. $\xi$, 
where $\nabla$ represents the gradient w.r.t. parameter $x$, 
and where the proximity operator of $r$ is the mapping on $\bR^N\to\bR^N$ defined by
$$
\prox_{\gamma r} : x\mapsto \arg\min_{y\in \bR^N} \left(\gamma\, r(y) + \frac{\|y-x\|^2}2\right)\,.
$$
The drift $g_\gamma$ has the form~\eqref{eq:drift-integral} where 
$h_\gamma(\xi,x) =  \gamma^{-1}(\prox_{\gamma r}(x-\gamma \nabla \ell(\xi,x)) - x)$
and $\mu$ represents the distribution of the r.v. $\xi$.
Under adequate hypotheses, we prove that the interpolated process converges narrowly to the solutions to the DI
$$
\dot \sx(t) \in  -\nabla_x \bE_\xi(\ell(\xi,\sx(t))) - \partial r(\sx(t))\,,
$$
where $\partial r$ represents the subdifferential of a function $r$, defined by
$$
\partial r(x) \eqdef \left\{u\in \RN\,:\, \forall y\in \RN,\, r(y)\geq r(x)+\ps{u,y-x}\right\}
$$
at every point $x\in \RN$ such that $r(x)<+\infty$, and $\partial r(x)=\emptyset$ elsewhere.
We provide a sufficient condition under which the iterates~(\ref{eq:prox-gradient}) satisfy the 
Pakes-Has'minskii criterion, which in turn, allows to characterize the long-run behavior of the iterates.
\end{example}

\begin{example} \label{ex:fluid}  
{\sl Fluid limit of a system of parallel queues with priority.}
We consider a time slotted queuing system composed of $N$ queues. The 
following model is inspired from \cite{ayesta2013scheduling,gas-gau-12}.  We
denote by $y^{k}_n$ the number of users in the queue $k$ at time $n$.  We
assume that a random number of $A^{k}_{n+1} \in \bN$ users arrive in the queue
$k$ at time $n+1$. The queues are prioritized: the users of Queue $k$ can only
be served if all users of Queues $\ell$ for $\ell < k$ have been served.
Whenever the queue $k$ is non-empty and the queues $\ell$ are empty for all
$\ell<k$, one user leaves Queue $k$ with probability $\eta_k > 0$. Starting 
with $y^k_0 \in \bN$, we thus have 
\[ 
y^{k}_{n+1} = y^{k}_n + A^{k}_{n+1} -
B^{k}_{n+1}\1_{\{y^{k}_n>0,\,y_{n}^{k-1} =\cdots = y_{n}^1 = 0\}}\,, 
\]
where $B^{k}_{n+1}$ is a Bernoulli r.v.~with parameter $\eta_k$, and where
$\1_S$ denotes the indicator of an event $S$, equal to one on that set and to
zero otherwise. We assume that the process
$((A^{1}_{n},\dots,A^{N}_n,B^{1}_n,\dots,B^{N}_n), {n\in\bN^*})$ is iid, and 
that the random variables $A^{k}_{n}$ have finite second moments. We denote by
$\lambda_k\eqdef \bE(A^{k}_{n}) > 0$ the arrival rate in Queue $k$. 
Given a scaling parameter $\gamma > 0$ which is assumed to be small, we are 
interested in the \emph{fluid-scaled process}, defined as 
$x^{k}_n = \gamma y^{k}_n$. This process is subject to the dynamics:
\begin{align}
  \label{eq:queue}
  x^{k}_{n+1} &= x^{k}_n + \gamma\,A^{k}_{n+1} - \gamma\,B^{k}_{n+1}
  \1_{\{x^{k}_n>0,\,x^{k-1}_{n} =\cdots = x^{1}_{n} = 0\}}\,.
\end{align}
The Markov chain $x_n = (x^{1}_{n},\dots,x^{N}_n)$ admits the representation 
\eqref{eq:decomp-markov-drift}, where the drift $g_\gamma$ is defined on 
$\gamma\bN^N$, and is such that its $k$-th component $g_\gamma^{k}(x)$ is 
\begin{equation}
\label{gk-queue} 
g_\gamma^{k}(x) = 
  \lambda_k-\eta_k\1_{\{x^{k}>0,\,x^{k-1} =\cdots = x^{1} = 0\}} \,,
\end{equation} 
for every $k\in \{1,\dots,N\}$ and every $x=(x^1,\dots,x^N)$ in $\gamma\bN^N$.
Introduce the vector 
$\bs u_k\eqdef (\lambda_1,\cdots,\lambda_{k-1},\lambda_k-\eta_k,\lambda_{k+1},
\dots,\lambda_N)$ for all $k$. Let $\bR_+\eqdef[0,+\infty)$, and define the 
set-valued map on $\bR_+^N$ 
\begin{equation}
\label{Hqueue} 
\sH(x) \eqdef \left\{
  \begin{array}[h]{ll}
    \bs u_1 & \text{ if }x^{(1)}>0 \\
    \co (\bs u_1,\dots,\bs u_k) & \text{ if }x^1 =\cdots = x^{k-1} = 0 \ 
  \text{and} \ x^{k}>0\, , 
  \end{array}
\right.
\end{equation} 
where $\co$ is the convex hull. Clearly, $g_\gamma(x)\in \sH(x)$ for every 
$x\in\gamma\bN^N$. In~\cite[\S~3.2]{gas-gau-12}, it is shown that the DI 
$\dot \sx(t) \in \sH(\sx(t))$ has a unique solution. Our results imply the 
narrow convergence of the interpolated process to this solution, hence 
recovering a result of \cite{gas-gau-12}. More importantly, if the following 
stability condition 
\begin{equation}
  \label{eq:stability-queue}
  \sum_{k=1}^N \frac{\lambda_k}{\eta_k} < 1  
\end{equation}
holds, our approach allows to establish the tightness of the occupation measure
of the iterates $x_n$, and to characterize the long-run behavior of these 
iterates. We prove that in the long-run, the sequence $(x_n)$ converges to zero 
in the sense of~\eqref{eq:longrun}. The ergodic convergence in the sense
of~\eqref{eq:longrun-ergodic} can be also established with a small extra 
effort. 
\end{example}

\begin{example}
  {\sl Random monotone operators.}
As a second application, we consider the problem of finding a zero of a maximal monotone operator
$\sA:\RN\to 2^\RN$:
\begin{equation}
  \label{eq:find-zero}
\text{Find }x\text{ s.t. }0\in A(x)\,.
\end{equation}
We recall that a set-valued map $\sA:\RN\to 2^\RN$ is said monotone if for every $x$, $y$ in $\RN$,
and every $u\in \sA(x)$, $v\in \sA(y)$, $\ps{u-v,x-y}\geq 0$. 
The domain and the graph of 
$\sA$ are the respective subsets of $\RN$ and $\RN\times \RN$ defined as
$\dom(\sA) \eqdef \{ x \in \RN \, : \, \sA(x) \neq \emptyset \}$, 
and $\graph(\sA) \eqdef \{ (x,y) \in \RN\times \RN \, : \, y \in \sA(x) \}$.  
We denote by $\zer(\sA) \eqdef \{x\in \RN\,:\, 0\in \sA(x)\}$ the set of zeroes of $\sA$.
The operator $\sA$ is proper if $\dom(\sA)\neq\emptyset$.
A proper monotone operator
$\sA$ is said maximal if its graph $\graph(\sA)$ is a maximal element
in the inclusion ordering.
Denote by $I$ the identity operator, and by $\sA^{-1}$ the inverse of the
operator $\sA$, defined by the fact that $(x,y) \in \graph(\sA^{-1})
\Leftrightarrow (y,x) \in \graph(\sA)$.  It is well know that $\sA$ is maximal monotone if
and only if, for all $\gamma > 0$, the \emph{resolvent}
$\eqdef ( I + \gamma \sA )^{-1}$ is a contraction defined on the whole space
(in particular, it is single valued). 

Problem~\eqref{eq:find-zero} arises
in several applications such as convex optimization, variational inequalities, 
or game theory.
The celebrated \emph{proximal point algorithm} \cite{rockafellar1976monotone} 
generates the sequence $(u_n, n\in \bN)$ defined recursively as 
$u_{n+1} = (I+\gamma \sA)^{-1}(u_n)$. The latter sequence converges to a zero of the operator $\sA$, whenever such a zero exists.
Recent works (see \cite{bia-hac-16} and references therein) have been devoted to the special case 
where the operator $\sA$ is defined as the following selection integral
$$
\sA(x) = \int A(s,x)\mu(ds)\,,
$$
where $\mu$ is a probability on $\Xi$ and where $\{A(s, \cdot), s \in \Xi\}$ 
is a family of maximal monotone operators.
In this context, a natural algorithm for solving~(\ref{eq:find-zero}) is 
\begin{equation}
  \label{eq:ppa}
  x_{n+1} = (I+\gamma A(\xi_{n+1},\,.\,))^{-1}(x_n)
\end{equation}
where $(\xi_n, n\in \bN^*)$ is an iid sequence of r.v. whose law coincides 
with $\mu$.
The asymptotic behavior of~\eqref{eq:ppa} is analyzed in \cite{bia-16} under the assumption
that the step size $\gamma$ is decreasing with $n$. On the other hand, the results of the present paper
apply to the case where $\gamma$ is a constant which does not depend on $n$.
Here, the drift $g_\gamma$ has the form \eqref{eq:drift-integral} where the map
$-h_\gamma(s,x) = \gamma^{-1}( x-(I+\gamma A(s, \cdot))^{-1}(x))$ is the so-called \emph{Yosida regularization}
of the operator $A(s,\cdot)$ at $x$. As $\gamma\to 0$, it is well known that 
for every $x\in\dom(A(s,\cdot))$, $-h_\gamma(s,x)$ converges to the element of 
least norm in $A(s,\cdot)$ \cite{bau-com-livre11}.
Thanks to our results, it can be shown that under some hypotheses, the interpolated process converges narrowly to the unique solution 
to the DI
\begin{equation}
  \dot \sx(t) \in  -\int A(s, \sx(t))\mu(ds)\,,\label{eq:di-mm}
\end{equation}
and, under the Pakes-Has'minskii condition, that the iterates $x_n$ converge in the long run to the zeroes of $\sA$.
\end{example}

}

\rev{
\section{About the Literature}
\label{sec:literature}

When the drift $g_\gamma$ does not depend on $\gamma$ and is supposed to be a Lispchitz continuous map,
the long term behavior of the iterates $x_n$ in the small step size regime has been studied in the treatises
\cite{ben-met-pri-livre90,ben-(cours)99,kus-yin-(livre)03,bor-livre08,ben-hir-aap99} among
others. In particular, narrow convergence of the interpolated process to the solution of an 
Ordinary Differential Equation (ODE) is established. The authors of \cite{for-pag-99}
introduce a Pakes-Has'minskii criterion to study the long-run behavior of the iterates.

The recent interest in the stochastic approximation when the ODE is replaced
with a differential inclusion dates back to \cite{ben-hof-sor-05}, where
decreasing steps were considered. A similar setting is considered
in~\cite{fau-rot-10}. A Markov noise was considered in the recent manuscript
\cite{yaj-bha-(arxiv)16}.  We also mention \cite{fau-rot-13}, where the ergodic
convergence is studied when the so called weak asymptotic pseudo trajectory
property is satisfied. 
The case where the DI is built from maximal monotone operators is studied in
\cite{bia-16} and \cite{bia-hac-16}. 

Differential inclusions arise in many applications, which include
game theory (see \cite{ben-hof-sor-05,ben-hof-sor-(partII)06}, 
\cite{rot-san-siam13} and the references therein), convex optimization \cite{bia-hac-16},
queuing theory or wireless communications, where stochastic approximation algorithms 
with non continuous drifts are frequently used, and can be modelled by
differential inclusions~\cite{gas-gau-12}. 

Differential inclusions with a constant step were studied in \cite{rot-san-siam13}. 
The paper \cite{rot-san-siam13} extends previous results of \cite{ben-sch-00}
to the case of a DI. The key result established in \cite{rot-san-siam13} is that 
the cluster points of the collection of invariant measures of the Markov chain
are invariant for the flow associated with the DI. 
Prop. \ref{prop:cluster} of the present paper restates this result in a more general setting
and using a shorter proof, which we believe to have its own interest.
Moreover, the so-called GASP model studied by  \cite{rot-san-siam13}
does not cover certain applications, such as the ones provided in Section~\ref{sec:examples}, for instance.
In addition, \cite{rot-san-siam13} focusses on the case where the space is compact, which
circumvents the issue of stability and simplifies the mathematical arguments. 
However, in many situations, the compactness assumption does not hold, and sufficient conditions
for stability need to be formulated.
Finally, we characterize the asymptotic behavior of the iterates $(x_n)$
(as well as their Cesar\`o means) in the doubly asymptotic regime where $n\to\infty$ then $\gamma\to 0$.
Such results are not present in \cite{rot-san-siam13}.
}


\section{Background}
\label{sec-bgrd} 

\subsection{General Notations}

The notation $C(E,F)$ is used to denote the set of continuous functions from
the topological space $E$ to the topological space $F$.  The notation $C_b(E)$
stands for the set of bounded functions in $C(E,\bR)$.  We use the conventions
$\sup \emptyset = -\infty$ and $\inf \emptyset = +\infty$. Notation 
$\lfloor x\rfloor$ stands for the integer part of $x$.

Let $(E,d)$ be a metric space. For every $x\in E$ and $S\subset E$, we define
$d(x,S)=\inf\{d(x,y):y\in S\}$. We say that a sequence $(x_n, n\in\bN)$ on $E$
converges to $S$, noted $x_n\to_n S$ or simply $x_n\to S$, if $d(x_n,S)$ tends
to zero as $n$ tends to infinity.  For $\varepsilon > 0$, we define the
$\varepsilon$-neighborhood of the set $S$ as $S_\varepsilon \eqdef \{x\in
E:d(x,S)<\varepsilon\}$. The closure of $S$ is denoted by $\overline S$, and
its complementary set by $S^c$.
The characteristic function of $S$ is the function $\1_S:E\to\{0,1\}$ equal to
one on $S$ and to zero elsewhere.

Let $E={\mathbb R}^N$ for some integer $N\geq 1$.  
We endow the space $C(\bR_+,E)$ with the topology of uniform convergence on 
compact sets. 
The space $C(\bR_+,E)$ is metrizable by the distance $d$ defined for every $\mathsf x,\mathsf y\in C(\bR_+,E)$ by
\begin{equation}
d(\mathsf x,\mathsf y)\eqdef\sum_{n\in\mathbb N}2^{-n}
\left(1\wedge\!\! 
\sup_{t\in [0,n]}\|\mathsf x(t)-\mathsf y(t)\|\right)\,,
\label{eq:d}
\end{equation}
where $\|\cdot\|$ denotes the Euclidean norm in $E$.

\subsection{Random Probability Measures}

Let $E$ denote a metric space and let $\mcB(E)$ be its Borel $\sigma$-field.
We denote by $\cM(E)$ the set of probability measures on $(E,\mcB(E))$.
The support $\support(\nu)$ of a measure $\nu\in \cM(E)$ 
is the smallest closed set $G$ such that $\nu(G) = 1$.
We endow $\cM(E)$ with the topology of narrow convergence: 
a sequence $(\nu_n, n\in \bN)$ on $\cM(E)$ converges to
a measure $\nu\in\cM(E)$ (denoted $\nu_n\Rightarrow \nu$) if for every 
$f\in C_b(E)$, $\nu_n(f)\to\nu(f)$, where $\nu(f)$ is a shorthand for 
$\int f(x) \nu(dx)$. 
If $E$ is a Polish space, $\cM(E)$ is metrizable by the L\'evy-Prokhorov distance,
and is a Polish space as well.
A subset $\cG$ of $\cM(E)$ is said tight if for every $\varepsilon>0$, there exists a compact
subset $K$ of $E$ such that for all $\nu\in \cG$, $\nu(K)>1-\varepsilon$. 
By Prokhorov's theorem, $\cG$ is tight if and only if it is relatively compact in $\cM(E)$.

We denote by $\delta_a$ the Dirac measure at the point $a\in E$.
If $X$ is a random variable on some measurable space $(\Omega,\mcF)$ into $(E,\mcB(E))$,
we denote by $\delta_X:\Omega\to\cM(E)$ the measurable mapping defined by $\delta_X(\omega)=\delta_{X(\omega)}$.
If $\Lambda:(\Omega,\mcF)\to (\cM(E),\mcB(\cM(E)))$ is a random variable on the set of probability measures,
we denote by $\bE\Lambda$ the probability measure defined by 
$
(\bE\Lambda)(f)\eqdef\bE(\Lambda(f))\,,
$
for every $f\in C_b(E)$. 

\subsection{Set-Valued Mappings and Differential Inclusions}

A set-valued mapping $\sH : E\rightrightarrows F$ is a function on $E$ into the
set $2^F$ of subsets of $F$. The graph of $\sH$ is 
$\graph(\sH)\eqdef \{(a,b)\in E\times F:y\in \sH(a)\}$.  The domain of $\sH$ is
$\dom(\sH) \eqdef \{ a \in E \, : \, \sH(a) \neq \emptyset \}$. The mapping
$\sH$ is said proper if $\dom(\sH)$ is non-empty.  We say that $\sH$ is
single-valued if $\sH(a)$ is a singleton for every $a\in E$ (in which case we
handle $\sH$ simply as a function $\sH:E\to F$).

Let $\sH: E\rightrightarrows  E$ be a set-valued map on $E=\bR^N$, where $N$ is a positive integer.
Consider the 
differential inclusion: 
\begin{equation}
\dot \sx(t)\in \sH(\sx(t))\,.\label{eq:di}
\end{equation}
We say that an absolutely continuous mapping $\sx:\bR_+\to E$ 
is a solution to the differential inclusion 
with initial condition $a\in E$ 
if $\sx(0)=a$ and if (\ref{eq:di}) holds for almost every $t\in \bR_+$.
We denote by 
$$
\Phi_{\sH}:E\rightrightarrows C(\bR_+, E)
$$ 
the set-valued mapping such that for every $a\in E$, $\Phi_{\sH}(a)$ 
is set of solutions to~(\ref{eq:di}) with initial condition $a$.
We refer to $\Phi_\sH$ as the evolution system induced by $\sH$.
For every subset $A\subset E$, we define $\Phi_\sH(A) = \bigcup_{a\in A}\Phi_\sH(a)$.

A mapping $\sH:E\rightrightarrows E$ is said \emph{upper
  semi continuous} (usc) at a point $a_0 \in E$ if for every open set $U$
containing $\sH(a_0)$, there exists $\eta>0$, such that for every
$a\in E$, $\|a-a_0\|<\eta$ implies $\sH(a)\subset U$.
It is said usc if it is usc at every point
\cite[Chap.~1.4]{aub-cel-(livre)84}.
In the particular case where $\sH$ is usc with nonempty compact convex values
and satisfies the condition
\begin{equation}
  \label{eq:lin-growth}
\exists c>0,\ \forall a\in E,\  \sup\{\|b\|\,:b\in \sH(a)\}\leq c(1+\|a\|)\ ,
\end{equation}
then, $\dom(\Phi_\sH)=E$, see \emph{e.g.} \cite{aub-cel-(livre)84},
and moreover, $\Phi_\sH(E)$ is closed in the metric space $(C(\bR_+, E), d)$. 

\subsection{Invariant Measures of Set-Valued Evolution Systems}
\label{subsec-inv-mes}

Let $(E,d)$ be a metric space.
We define the shift operator $\Theta:C(\bR_+,E)\to
C(\bR_+,C(\bR_+,E))$ s.t. for every $\mathsf x\in C(\bR_+,E)$,
$\Theta(\mathsf x) : t\mapsto \mathsf x(t+\,\cdot\,)$.

Consider a set-valued mapping $\Phi:E\rightrightarrows C(\bR_+,E)$.
When $\Phi$ is single-valued (\emph{i.e.}, for all $a\in E$, $\Phi(a)$
is a continuous function), a measure $\pi\in \cM(E)$ is called an
\emph{invariant measure} for $\Phi$, or $\Phi$-invariant, if for all
$t>0$, $\pi=\pi\Phi_t^{-1}$, where $\Phi_t:E\to E$ is the map defined
by $\Phi_t(a) = \Phi(a)(t)$.  For all $t\geq 0$, we define the
projection $p_t:C(\bR_+,E)\to E$ by $p_t(\sx) = \sx(t)$.

The definition can be extended as follows to the case where $\Phi$ is set-valued.
\begin{definition}
\label{def-inv} 
A probability measure $\pi\in\cM(E)$ is said invariant for $\Phi$
if there exists $\upsilon\in\cM(C(\bR_+,E))$ s.t.
\begin{enumerate}[(i)]
\item\label{inv-support} $\support(\upsilon)\subset \overline{\Phi(E)}$\,;
\item\label{inv-Theta} $\upsilon$ is $\Theta$-invariant\,;
\item\label{inv-margin} $\upsilon p_0^{-1}=\pi$.
\end{enumerate}
\end{definition} 
When $\Phi$ is single valued, both definitions coincide.
The above definition is borrowed from \cite{fau-rot-13} (see also 
\cite{mil-aki-99}). 
Note that $\overline{\Phi(E)}$ can be replaced by $\Phi(E)$ whenever the latter
set is closed (sufficient conditions for this have been provided above).

The limit set of a function $\sx \in C(\bR_+, E)$ is defined as 
\[
L_\sx \eqdef \bigcap_{t\geq 0}\overline{\sx([t,+\infty)} \,.
\]
It coincides with the set of points of the form $\lim_n \sx(t_n)$ for some 
sequence $t_n\to\infty$. Consider now a set valued mapping 
$\Phi:E\rightrightarrows C(\bR_+,E)$. The limit set $L_{\Phi(a)}$ of a point 
$a \in E$ for $\Phi$ is 
\[
L_{\Phi(a)} \eqdef \bigcup_{\sx \in \Phi(a)} L_\sx \, ,
\]
and $L_\Phi \eqdef \bigcup_{a\in E}L_{\Phi(a)}$.
A point $a$ is said recurrent for $\Phi$ if $a \in L_{\Phi(a)}$.
The Birkhoff center of $\Phi$ is the closure 
of the set of recurrent points
\[
\text{BC}_{\Phi} \eqdef 
\overline{\{ a \in E \, : \, a \in L_{\Phi(a)} \}}\, . 
\]
The following result, established in \cite{fau-rot-13} 
(see also \cite{aub-fra-las-91}), is a consequence of the celebrated 
recurrence theorem of Poincar\'e. 
\begin{proposition} 
\label{poincare} 
Let $\Phi:E\rightrightarrows C(\bR_+,E)$. Assume that $\Phi(E)$ is closed. Let $\pi\in \cM(E)$ be an invariant measure for $\Phi$. Then,
$\pi(\text{BC}_\Phi) = 1$. 
\end{proposition}
We denote by $\cI(\Phi)$ the subset of $\cM(E)$ formed by all invariant
measures for $\Phi$. We define
$$
\mcI(\Phi) \eqdef \{ \mathfrak{m}\in\cM(\cM(E))\,:\,\forall A\in \mcB(\cM(E)),\, \cI(\Phi)\subset A\,\Rightarrow\,\mathfrak{m}(A)=1\}\,.
$$
We define the mapping $\aver:C(\bR_+,E)\to C(\bR_+,E)$ by 
$$
\aver(\sx) :t\mapsto \frac 1t \int_0^t \sx(s) \, ds \,,
$$
and $\aver(\sx)(0)=\sx(0)$. Finally, we define 
$\aver(\Phi) : E \rightrightarrows C(\bR_+, E)$ by
$\aver(\Phi)(a) = \{ \aver(\sx) \, : \, \sx \in \Phi(a) \}$ for
each $a\in E$.

\subsection{The Selection Integral}

 Let $(\Xi, \mcG,\mu)$ denote an arbitrary probability space. 
For $1 \leq p < \infty$, we denote by
${\mathcal L}^p(\Xi, {\mcG}, \mu;  E)$ the Banach space of the measurable
functions $\varphi : \Xi \to  E$ such that $\int \| \varphi \|^p d\mu <
\infty$. 
For any set-valued mapping $G:\Xi\rightrightarrows  E$, we define the set
\[
\Selec^p_G \eqdef 
\{ \varphi \in {\mathcal L}^p(\Xi, {\mcG}, \mu;  E) \, : \, 
\varphi(\xi) \in G(\xi) \ \mu-\text{a.e.} \} \, .
\] 
Any element of $\Selec^1_G$ is referred to as an \emph{integrable selection}.
If $\Selec^1_G \neq \emptyset$, the mapping $G$ is said to be 
integrable.  The \emph{selection integral} \cite{molchanov2006theory} of $G$ is the set
\[
\int G d\mu \eqdef \overline{\left\{ \int_\Xi \varphi d\mu \ : \ 
  \varphi \in \Selec^1_G \right\}} \,.
\]

\section{Main Results}
\label{sec-main}

\subsection{Dynamical Behavior}
\label{rm-general} 

From now on to the end of this paper, we set $E\eqdef\bR^N$ where $N$
is a positive integer. \rev{Choose $\gamma_0>0$. For every
  $\gamma\in (0,\gamma_0)$, we introduce a probability transition
  kernel $P_\gamma$ on $E\times\mcB(E)\to [0,1]$.}
Let $(\Xi, \mcG,\mu)$ be an arbitrary probability space.  

\begin{assumption*}[RM] There exist a $\mcG\otimes\mcB(E)/\mcB(E)$-measurable map 
$h_\gamma: \Xi\times E \to E$ 
and $H: \Xi\times E\rightrightarrows E$ such that:
\begin{enumerate}[i)]
\item \label{hyp:drift} \rev{For every $x\in E$, $$\int \frac{y-x}\gamma P_\gamma(x,dy) = \int h_\gamma(s,x)\mu(ds)\,.$$}
\item\label{hyp:RM-cvg}  
For every $s$ $\mu$-a.e. and for every converging sequence
$(u_n,\gamma_n)\to (u^\star,0)$ on $E\times (0,\gamma_0)$,
$$h_{\gamma_n}(s,u_n)\to H(s,u^\star)\,.$$
\item For all $s$ $\mu$-a.e., $H(s,\cdot)$ is proper, usc, 
with closed convex values.
\item\label{hyp:RM-integ}
For every $x\in E$, $H(\cdot, x)$ is $\mu$-integrable. We set 
$\sH(x)\eqdef  \int H(s, x)\, \mu(ds)$. 
\item \label{hyp:flot-borne} For every $T>0$ and every compact set 
$K\subset E$, 
$$
\sup\{\|\sx(t)\|: t\in [0,T], \sx\in \Phi_{\sH}(a), a\in K\}<\infty\,.
$$ 
\item \label{hyp:RM-moments} \rev{For every compact set $K\subset E$, there exists $\epsilon_{K}>0$ 
such that
    \begin{align}
&\sup_{x\in K}\sup_{0<\gamma<\gamma_0}
 \int\left\|\frac{y-x}\gamma\right\|^{1+\epsilon_K} P_\gamma(x,dy) <\infty\,,
\label{eq:moment}\\
&\sup_{x\in K}\sup_{0<\gamma<\gamma_0} \int \left\|h_\gamma(s,x)\right\|^{1+\epsilon_K}\mu(ds)<\infty\,.
\label{eq:moment-bis}
\end{align}}
\end{enumerate}
\end{assumption*}
\rev{Assumption \ref{hyp:drift} implies that the drift has the form~(\ref{eq:drift}). 
As mentioned in the introduction, this is for instance useful in the case of iterative Markov models
such as (\ref{eq:iterative-model}).}
Assumption \ref{hyp:flot-borne} requires implicitly that the set of solutions
$\Phi_{\sH}(a)$ is non-empty for any value of $a$. It holds true if, 
\emph{e.g.}, the linear growth condition~\eqref{eq:lin-growth} on $\sH$ is
satisfied. 

On the canonical space $\Omega\eqdef E^{\mathbb N}$ equipped with the
$\sigma$-algebra $\mcF\eqdef \mcB(E)^{\otimes \bN}$,
we denote by $X:\Omega\to E^{\bN}$ the canonical process
defined by $X_n(\omega)=\omega_n$ for every $\omega=(\omega_k,k\in \bN)$ and every $n\in \bN$, where 
$X_n(\omega)$ is the $n$-th coordinate of $X(\omega)$.
For every $\nu\in \cM(E)$ and $\gamma\in (0,\gamma_0)$, 
we denote by $\bP^{\nu,\gamma}$ the unique probability measure on $(\Omega,\mcF)$
such that $X$ is an homogeneous Markov chain with initial distribution $\nu$ and transition kernel $P_\gamma$.
We denote by $\bE^{\nu,\gamma}$ the corresponding expectation. When $\nu=\delta_a$ for some $a\in E$, 
we shall prefer the notation $\bP^{a,\gamma}$ to $\bP^{\delta_a,\gamma}$.

The set $C(\bR_+,E)$ is equipped with the topology of uniform convergence 
on the compact intervals, who is known to be compatible with the distance 
$d$ defined by~\eqref{eq:d}. For every $\gamma>0$,
we introduce the measurable map on  
$(\Omega,\mcF)\to (C(\bR_+,E),\mcB(C(\bR_+,E)))$, such that for every 
$x=(x_n,n\in \bN)$ in $\Omega$,
$$
  \sX_\gamma(x)\,:t \mapsto x_{\lfloor \frac t\gamma\rfloor} + (t/\gamma-\lfloor t/\gamma\rfloor)(x_{\lfloor \frac t\gamma\rfloor+1}-x_{\lfloor \frac t\gamma\rfloor})
\,.
$$
The random variable $\sX_\gamma$ will be referred to as the linearly 
\emph{interpolated process}.
On the space $(C(\bR_+,E),\mcB(C(\bR_+,E)))$, the distribution of the 
r.v.~$\sX_\gamma$ is $\bP^{\nu,\gamma}\sX_\gamma^{-1}$.
\begin{theorem}
\label{th:SA=wAPT} 
Suppose that Assumption (RM) is satisfied. Then, for every compact set 
$K\subset E$, the family 
$\{ \bP^{a,\gamma}\sX_\gamma^{-1}:a\in K,0<\gamma<\gamma_0 \}$ is tight. 
Moreover, for every $\varepsilon>0$, 
\[
\sup_{a\in K}\,\bP^{a,\gamma}\left(d(\sX_\gamma,\Phi_\sH(K))>\varepsilon\right)\xrightarrow[\gamma\to 0]{}0\,.
\]
\end{theorem}

\subsection{Convergence Analysis}

For each $\gamma\in (0,\gamma_0)$, we denote by
$$
\cI(P_\gamma) \eqdef \{\pi\in\cM(E)\,:\,\pi = \pi P_\gamma\}
$$
the set of invariant probability measures of $P_\gamma$.  Letting 
$\cP = \{P_\gamma,0<\gamma<\gamma_0\}$, we define 
$\cI(\cP) = \bigcup_{\gamma\in (0,\gamma_0)}\cI(P_\gamma)$.  
We say that a measure $\nu\in \cM(E)$ is a cluster point of $\cI(\cP)$
as $\gamma\to 0$, if there exists a sequence $\gamma_j\to 0$ and a sequence
of measures $(\pi_j,j\in \bN)$ s.t. $\pi_j\in \cI(P_{\gamma_j})$ for all $j$,
and  $\pi_j\Rightarrow \nu$.

We define
$$
\mcI(P_\gamma) \eqdef \{\mathfrak{m}\in\cM(\cM(E))\,:\, \support(\mathfrak{m})\subset \cI(P_\gamma)\}\,,
$$
and $\mcI(\cP) =\bigcup_{\gamma\in (0,\gamma_0)}\mcI(P_\gamma)$. 
We say that a measure $\mm\in \cM(\cM(E))$ is a cluster point of $\mcI(\cP)$
as $\gamma\to 0$, if there exists a sequence $\gamma_j\to 0$ and a sequence
of measures $(\mm_j,j\in \bN)$ s.t. $\mm_j\in \mcI(P_{\gamma_j})$ for all $j$,
and  $\mm_j\Rightarrow \mm$.

\begin{proposition}
\label{prop:cluster}
  Suppose that Assumption (RM) is satisfied. Then, 
  \begin{enumerate}[i)]
  \item\label{cluster-Idroit}
   As $\gamma\to 0$, any cluster point of $\cI(\cP)$ is an element of $\cI(\Phi_\sH)$;
  \item\label{cluster-Icursif}
  As $\gamma\to 0$, any cluster point of $\mcI(\cP)$ is an element of 
  $\mcI(\Phi_\sH)$.
  \end{enumerate}
\end{proposition}
In order to explore the consequences of this proposition, we introduce two
supplementary assumptions. The first is the so-called Pakes-Has'minskii 
tightness criterion, who reads as follows \cite{for-pag-99}: 
\begin{assumption*}[PH] 
  There exists measurable mappings $V:E\to[0,+\infty)$, $\psi:E\to[0,+\infty)$
and two functions $\alpha:(0,\gamma_0)\to(0,+\infty)$, $\beta:(0,\gamma_0)\to\bR$, such that 
  \begin{align*}
    & \sup_{\gamma\in (0,\gamma_0)}\frac{\beta(\gamma)}{\alpha(\gamma)}<\infty 
    \qquad\text{and}\qquad {\lim_{\|x\|\to+\infty}\psi(x)=+\infty}\,,
  \end{align*}
and for every $\gamma\in (0,\gamma_0)$,
$$
P_\gamma V\leq V-\alpha(\gamma)\psi+\beta(\gamma)\,.
$$
\end{assumption*}
We recall that a transition kernel $P$ on $E\times\mcB(E)\to [0,1]$ 
is said \emph{Feller} if the mapping $Pf:x\mapsto \int f(y)P(x,dy)$ 
is continuous for any $f\in C_b(E)$. 
If $P$ is Feller, then the set of invariant measures of $P$ 
is a closed subset of $\cM(E)$.
The following assumption ensures that for all $\gamma\in (0,\gamma_0)$,
$P_\gamma$ is Feller.
\begin{assumption*}[FL]
  For every $s\in \Xi$, $\gamma\in (0,\gamma_0)$, the function 
$h_\gamma(s,\cdot)$ is continuous.
\end{assumption*}

\begin{theorem}
  Let Assumptions (RM), (PH) and (FL) be satisfied. Let $\psi$ and $V$ be the functions
specified in (PH). Let $\nu\in\cM(E)$ s.t. $\nu(V)<\infty$.
Let $\mU \eqdef \bigcup_{\pi \in \cI(\Phi)} \support(\pi)$. Then, for all 
$\varepsilon > 0$, 
\begin{equation}
\limsup_{n\to\infty} \frac 1{n+1}\sum_{k=0}^n \bP^{\nu,\gamma}( d(X_k,\mU)>\varepsilon)\xrightarrow[\gamma\to 0]{}0\,. 
\label{eq:support}
\end{equation}
Let $N'\in\bN^*$ and $f\in C(E,\bR^{N'})$. Assume that there exists $M\geq 0$ and $\varphi:\bR^{N'}\to\bR_+$ such that 
$\lim_{\|a\|\to\infty}\varphi(a)/{\|a\|}=\infty$ and
\begin{equation}
\forall a\in E,\ \varphi(f(a))\leq M(1+\psi(a))\,.\label{eq:dVP}
\end{equation}
Then, the set $\cS_f\eqdef\{\pi(f)\,:\, \pi\in\cI(\Phi)\text{
     and } \pi(\|f(\cdot)\|)<\infty\}$ is nonempty. For all $n\in\bN$, $\gamma\in (0,\gamma_0)$, the r.v.
$$
F_n \eqdef \frac 1{n+1}\sum_{k=0}^n f(X_k)
$$
is $\bP^{\nu,\gamma}$-integrable, and satisfies for all $\varepsilon>0$,
\begin{align}
&\limsup_{n\to\infty}\ d\left(\bE^{\nu,\gamma}( F_n)\,,\cS_f\right)\xrightarrow[\gamma\to 0]{}0\,,\label{eq:CVSE} \\
&\limsup_{n\to\infty}\ \bP^{\nu,\gamma}\left(d\left(F_n\,,\cS_f\right)\geq \varepsilon\right)\xrightarrow[\gamma\to 0]{}0\,. \label{eq:CVSf} 
\end{align}
\label{the:CV}
\end{theorem} 

\begin{theorem}
\label{cvg-CVSI} 
  Let Assumptions (RM), (PH) and (FL) be satisfied. Assume that $\Phi_\sH(E)$ is closed. 
Let $\psi$ and $V$ be the functions specified in (PH). Let $\nu\in\cM(E)$ s.t. $\nu(V)<\infty$.
Assume that
$$
\lim_{\|a\|\to\infty}\frac{\psi(a)}{\|a\|}=+\infty\,.
$$
For all $n\in \bN$, define 
$
\overline X_n\eqdef \frac 1{n+1} \sum_{k=0}^n X_k\,.
$
Then, for all $\varepsilon > 0$, 
\begin{align*}
&  \limsup_{n\to\infty}\ d\left(\bE^{\nu,\gamma}(\overline X_n)\,,\co(L_{\aver(\Phi)})\right)\xrightarrow[\gamma\to 0]{}0\,,\\
&\limsup_{n\to\infty}\ \bP^{\nu,\gamma}\Bigl(d\Bigl(\overline X_n\,,\co(L_{\aver(\Phi)})\Bigr)\geq \varepsilon\Bigr)
 \xrightarrow[\gamma\to 0]{}0\, , 
\end{align*}
where $\co(S)$ is the convex hull of the set $S$. 
\end{theorem}

\begin{theorem}
\label{cvg-XY}
  Let Assumptions (RM), (PH) and (FL) be satisfied.  Assume that $\Phi_\sH(E)$ is closed. Let $\psi$ and $V$ be the functions
specified in (PH). Let $\nu\in\cM(E)$ s.t. $\nu(V)<\infty$.
Then, for all $\varepsilon > 0$, 
$$
\limsup_{n\to\infty}\ \frac 1{n+1}\sum_{k=0}^n 
 \bP^{\nu,\gamma}\left(d\left(X_k\,,\text{BC}_{\Phi} \right) 
   \geq \varepsilon\right) \xrightarrow[\gamma\to 0]{}0\,.
$$
\end{theorem}

\section{Proof of Theorem~\ref{th:SA=wAPT}}
\label{sec-prf-narrow}

The first lemma is a straightforward adaptation of the \emph{convergence theorem} \cite[Chap.~1.4, Th.~1, pp. 60]{aub-cel-(livre)84}.
Hence, the proof is omitted.
We denote by $\lambda_T$ the Lebesgue measure on $[0,T]$.
\begin{lemma}
  \label{lem:cvth}
Let $\{F_\xi:\xi\in\Xi\}$ be a family of mappings on $ E\rightrightarrows E$.
Let $T>0$ and for all $n\in\bN$, let 
$u_n:[0,T]\to E$, $v_n:\Xi\times[0,T]\to E$ be measurable maps
w.r.t $\mcB([0,T])$ and $\mcG \otimes \mcB([0,T])$ respectively. Note for simplicity
$\cL^1\eqdef \cL^1(\Xi\times [0,T],\mcG\otimes \mcB([0,T]),\mu\otimes \lambda_T;\bR)$.
Assume the following.
\begin{enumerate}[i)]
\item \label{hyp:cvToGraph} For all $(\xi,t)$ $\mu\otimes\lambda_T$-a.e., 
$(u_n(t),v_n(\xi,t))\to_n \mathrm{gr}(F_\xi)$. 
\item $(u_n)$ converges $\lambda_T$-a.e. to a function $u:[0,T]\to E$.
\item \label{hyp:cvL1} For all $n$, $v_n\in \cL^1$ and converges weakly in $\cL^1$ to a 
function $v:\Xi\times [0,T]\to E$.
\item For all $\xi$ $\mu$-a.e., $F_\xi$ is proper upper semi continuous
with closed convex values.
\end{enumerate}
Then, for all $(\xi,t)$ $\mu\otimes \lambda_T$-a.e., $v(\xi,t)\in F_\xi(u(t))$.
\end{lemma}
Given $T > 0$ and $0<\delta\leq T$, we denote by 
\[ 
w_\sx^T(\delta) \eqdef 
\sup\{\|\sx(t)-\sx(s)\| :|t-s|\leq \delta, (t,s)\in [0,T]^2\}
\] 
the modulus of continuity on $[0,T]$ of any $\sx\in C(\bR_+, E)$. 
\begin{lemma}
\label{h-ui} 
For all $n\in \bN$, denote by $\mcF_n\subset\mcF$ the $\sigma$-field generated 
by the r.v. $\{X_{k},0\leq k\leq n\}$. For all $\gamma\in (0,\gamma_0)$, define 
$Z^\gamma_{n+1} \eqdef \gamma^{-1}(X_{n+1}-X_{n})$.
Let $K\subset E$ be compact. Let 
$\{\bar \bP^{a,\gamma},a\in K,0<\gamma<\gamma_0\}$ be a family of probability 
measures on $(\Omega,\mcF)$ satisfying the following uniform integrability 
condition:
\begin{equation}
\sup_{{n\in\bN^*, a\in K, \gamma\in (0,\gamma_0)}} \bar \bE^{a,\gamma}(\|Z^\gamma_{n}\|\1_{\|Z^\gamma_{n}\|>A})\xrightarrow[]{A\to+\infty} 0\,.\label{eq:ui}
\end{equation}
Then, $\{\bar P^{a,\gamma}\sX_\gamma^{-1}:a\in K,0<\gamma<\gamma_0\}$ is 
tight. Moreover, for any $T > 0$, 
  \begin{equation}
\sup_{a\in K}\bar\bP^{a,\gamma}\left(\max_{0\leq n\leq \lfloor\frac T\gamma\rfloor}\gamma\left\|\sum_{k=0}^n\left(Z^\gamma_{k+1}-\bar\bE^{a,\gamma}(Z^\gamma_{k+1}|\mcF_k)\right)
\right\|>\varepsilon\right)\xrightarrow[]{\gamma\to 0} 0\,.
\label{eq:asym-rate-change}
\end{equation}
\label{lem:tightC}
\end{lemma}
\begin{proof}
We prove the first point. Set $T>0$, let $0<\delta\leq T$, and choose 
$0\leq s\leq t\leq T$ s.t. $t-s\leq \delta$. Let $\gamma\in(0,\gamma_0)$ and 
set $n\eqdef \lfloor \frac t\gamma\rfloor$, 
$m\eqdef \lfloor \frac s\gamma\rfloor$.  For any $R>0$,
$$
\|\sX_\gamma(t)-\sX_\gamma(s)\|\leq \gamma\!\!\sum_{k=m+1}^{n+1}\!\!\|Z_k^{\gamma}\|\leq \gamma(n-m+1)R + \gamma\!\!\sum_{k=m+1}^{n+1}\!\!\|Z_k^{\gamma}\|\1_{\|Z_k^{\gamma}\|>R}\,.
$$
Noting that $n-m+1\leq \frac{\delta}\gamma$ and using Markov inequality, we obtain
\begin{align*}
  \bar\bP^{a,\gamma}\sX_\gamma^{-1}(\{x:w_x^T(\delta)>\varepsilon\}) &\leq  \bar\bP^{a,\gamma}\left(\gamma\!\!\sum_{k=1}^{ \lfloor \frac T\gamma\rfloor+1}\!\!\|Z_k^{\gamma}\|\1_{\|Z_k^{\gamma}\|>R}>\varepsilon-\delta R\right) \\
&\leq T \frac{\sup_{k\in\bN^*} \bar \bE^{a,\gamma}\left(\|Z_k^{\gamma}\|\1_{\|Z_k^{\gamma}\|>R}\right)}{\varepsilon-\delta R}\,,
\end{align*}
provided that $R\delta<\varepsilon$. Choosing $R=\varepsilon/(2\delta)$ and 
using the uniform integrability, 
\[
\sup_{a\in K,0<\gamma<\gamma_0}
\bar\bP^{a,\gamma}\sX_\gamma^{-1}(\{x:w_x^T(\delta)>\varepsilon\})
\xrightarrow[]{\delta\to 0} 0\, . 
\]
As $\{\bar\bP^{a,\gamma}\sX_\gamma^{-1}p_0^{-1},0<\gamma<\gamma_0,a\in K\}$ 
is obviously tight, the tightness of 
$\{\bar P^{a,\gamma}\sX_\gamma^{-1},a\in K,0<\gamma<\gamma_0\}$ follows 
from \cite[Theorem 7.3]{bil-(livre)99}.

We prove the second point. 
We define $M^{a,\gamma}_{n+1} \eqdef \sum_{k=0}^n\left(Z^\gamma_{k+1}-\bar\bE^{a,\gamma}(Z^\gamma_{k+1}|\mcF_k)\right)$.
We introduce
\begin{align*}
  \eta_{n+1}^{a,\gamma,\leq} &\eqdef  Z_{n+1}^{\gamma}\1_{\|Z_{n+1}^{\gamma}\|\leq R}-\bar\bE^{a,\gamma}\left(Z_{n+1}^{\gamma}\1_{\|Z_{n+1}^{\gamma}\|\leq R}|\mcF_n\right)
\end{align*}
and we define $\eta_{n+1}^{a,\gamma,>}$ in a similar way, by replacing $\leq$ with $>$ in the right hand side of the above equation.
Clearly, for all $a\in K$, $M_{n+1}^{a,\gamma} =  \eta_{n+1}^{a,\gamma,\leq}+ \eta_{n+1}^{a,\gamma,>}$. Thus,
$$
\gamma\left\|M^{a,\gamma}_{n+1}\right\|\leq \|S_{n+1}^{a,\gamma,\leq}\| + \|S_{n+1}^{a,\gamma,>}\| 
$$
where $S_{n+1}^{a,\gamma,\leq}\eqdef \gamma \sum_{k=0}^n\eta_{k+1}^{a,\gamma,\leq}$ and 
$S_{n+1}^{a,\gamma,>}$ is defined similarly. 
Under $\bar\bP^{a,\gamma}$, the random processes $S^{a,\gamma,\leq}$ and $S^{a,\gamma,>}$ are $\mcF_n$-adapted martingales.
Defining $q_\gamma\eqdef \lfloor\frac T\gamma\rfloor+1$, we obtain by Doob's 
martingale inequality and by the boundedness of the increments of 
$S_{n}^{a,\gamma,\leq}$ that  
$$ 
\bar\bP^{a,\gamma}\left(\max_{1\leq n\leq q_\gamma}\|S_{n}^{a,\gamma,\leq}\|
      >\varepsilon\right) \leq 
\frac{\bar\bE^{a,\gamma} (\|S_{q_\gamma}^{a,\gamma,\leq}\|)}{\varepsilon} 
\leq \frac{\bar\bE^{a,\gamma} (\|S_{q_\gamma}^{a,\gamma,\leq}\|^2)^{1/2}}{\varepsilon} 
\leq \frac 2\varepsilon \gamma R \sqrt
{q_\gamma}\,, 
$$
and the right hand side tends to zero uniformly in $a\in K$ as $\gamma\to 0$. By the same inequality,
$$
\bar\bP^{a,\gamma}\left(\max_{1\leq n\leq q_\gamma}\|S_{n}^{a,\gamma,>}\|>\varepsilon\right) \leq \frac 2\varepsilon q_\gamma \gamma \sup_{k\in\bN^*} \bar\bE^{a,\gamma}\left(\|Z_k^{\gamma}\|\1_{\|Z_k^{\gamma}\|>R}\right)\,.
$$
Choose an arbitrarily small $\delta>0$ and select $R$ as large as need
in order that the supremum in the right hand side is no larger than
$\varepsilon \delta/(2T+2\gamma_0)$. Then the left hand side is no larger than
$\delta$. Hence, the proof is concluded.
\end{proof}

For any $R>0$, define 
$h_{\gamma,R}(s,a)\eqdef h_{\gamma}(s,a)\1_{\|a\|\leq R}$.  Let 
$H_R(s,x)\eqdef H(s,x)$ if $\|x\|<R$, $\{0\}$ if $\|x\|>R$, and $E$ otherwise.
Denote the corresponding selection integral as 
$\sH_R(a) = \int H_R(s,a)\, \mu(ds)$. Define 
$\tau_R(x)\eqdef\inf\{n\in \bN:x_n>R\}$ for all $x\in \Omega$. We also 
introduce the measurable mapping $B_R:\Omega\to \Omega$, given by
$$
B_R(x) : n\mapsto x_n\1_{n< \tau_R(x)} + x_{\tau_R(x)}\1_{n\geq \tau_R(x)}
$$
for all $x\in \Omega$ and all $n\in \bN$. 
\begin{lemma}
\label{lem:aptTrunk}
Suppose that Assumption (RM) is satisfied. Then, for every compact set 
$K\subset E$, the family 
$\{\bP^{a,\gamma}B_R^{-1}\sX_\gamma^{-1},\gamma\in (0,\gamma_0), a\in K\}$ 
is tight. Moreover, for every $\varepsilon>0$, 
\[
\sup_{a\in K}\,\bP^{a,\gamma}B_R^{-1}\left[d(\sX_\gamma,\Phi_{\sH_R}(K))>\varepsilon\right]\xrightarrow[\gamma\to 0]{}0\,.
\] 
\end{lemma}
\begin{proof}
We introduce the measurable mapping $\Delta_{\gamma,R}:\Omega\to E^\bN$ s.t. for all $x\in \Omega$,
$\Delta_{\gamma,R}(x)(0)\eqdef 0$ and
$$
\Delta_{\gamma,R}(x)(n)\eqdef (x_n-x_{0}) - \gamma \sum_{k = 0}^{n-1} \int h_{\gamma,R}(s,x_k)\mu(ds)
$$
for all $n\in \bN^*$. We also introduce the measurable mapping $\sG_{\gamma,R}:C(\bR_+,E)\to C(\bR_+,E)$ 
s.t. for all $\sx\in C(\bR_+,E)$, 
$$
\sG_{\gamma,R}(\sx)(t) \eqdef 
\rev{\int_0^t\int h_{\gamma,R}(s, \sx(\gamma \lfloor  u/\gamma\rfloor)) \, \mu(ds)\,du\, .} 
$$
We first express the interpolated process in integral form. For every $x\in E^\bN$ and $t\geq 0$,
$$
\sX_\gamma(x)(t)= x_0 
+ \int_0^t \gamma^{-1}(x_{\lfloor \frac u\gamma\rfloor +1}
 -x_{\lfloor \frac u\gamma\rfloor})\, du \,,
$$
from which we obtain the decomposition
\begin{equation}
\sX_\gamma(x) = x_0 + \sG_{\gamma,R}\circ \sX_\gamma(x)  + \sX_\gamma\circ \Delta_{\gamma,R}(x)\,. \label{eq:expr-itp}
\end{equation}
\rev{ The uniform integrability condition~(\ref{eq:ui})
is satisfied when letting $\bar\bP^{a,\gamma}\eqdef \bP^{a,\gamma}B_R^{-1}$. Indeed, 
\begin{align*}
  \bar\bE^{a,\gamma}(\|\gamma^{-1}(X_{n+1}-X_n)\|^{1+\epsilon_K}) &= \bE^{a,\gamma}\left(\left\|\gamma^{-1}(X_{n+1}-X_n)\right\|^{1+\epsilon_K} \1_{\tau_R(X)>n}\right) \,,
\end{align*}
and the condition~\eqref{eq:ui} follows from hypothesis~(\ref{eq:moment}).
On the other hand, as the event $\{\tau_R(X)>n\}$ is $\mcF_n$-measurable,
\begin{align*}
  \bar \bE^{a,\gamma}(\gamma^{-1}(X_{n+1}-X_n)|\mcF_n) &= \bE^{a,\gamma}(\gamma^{-1}(X_{n+1}-X_n)|\mcF_n)  \1_{\tau_R(X)>n} \\
 &= \int h_\gamma(s,X_n)\mu(ds) \1_{\tau_R(X)>n}  \\
 &= \int h_{\gamma,R}(s,X_n)\mu(ds)  \,.
\end{align*}
}Thus, Lemma~\ref{lem:tightC} implies that for all $\varepsilon>0$ and 
$T>0$,
$$
\sup_{a\in K}\bar\bP^{a,\gamma}\left(
\max_{0\leq n\leq \lfloor\frac T\gamma\rfloor}
 \gamma\left\|\Delta_{\gamma,R}(x)(n+1)\right\|>\varepsilon\right) 
\xrightarrow[]{\gamma\to 0} 0\,.
$$
It is easy to see that for all $x\in \Omega$, the function 
$\sX_\gamma\circ \Delta_{\gamma,R}(x)$ is bounded on every compact interval 
$[0,T]$ by $\max_{n\leq \lfloor  T/\gamma\rfloor}\|\Delta_{n+1}^{\gamma}\|$. 
This in turns leads to:
\begin{equation}
\sup_{a\in K}\bar \bP^{a,\gamma}(\|\sX_\gamma\circ \Delta_{\gamma,R}\|_{\infty,T}>\varepsilon) \xrightarrow[]{\gamma\to 0}0\,,\label{eq:rate-of-change-continuous}
\end{equation}
where the notation $\|\sx\|_{\infty,T}$ stands for the uniform norm of $\sx$ on $[0,T]$.

As a second consequence of Lemma~\ref{lem:tightC}, the family 
$\{\bar\bP^{a,\gamma}\sX_\gamma^{-1},0<\gamma<\gamma_0,a\in K\}$ is tight. 
Choose any subsequence $(a_n,\gamma_n)$ s.t. $\gamma_n\to 0$ and $a_n\in K$.
Using Prokhorov's theorem and the compactness of $K$, there exists a subsequence
(which we still denote by $(a_n,\gamma_n)$) and there exist some $a^\star\in K$ and some 
 $\upsilon\in \cM(C(\bR_+,E))$ such that  $a_n\to a^\star$ and
$\bar\bP^{a_n,\gamma_n}\sX_{\gamma_n}^{-1}$ converges narrowly to~$\upsilon$.
By Skorokhod's representation theorem, we introduce some r.v. ${\mathsf z}$, $\{\sx_n,n\in\bN\}$ on $C(\bR_+, E)$
with respective distributions $\upsilon$ and $\bar\bP^{a_n,\gamma_n}\sX_{\gamma_n}^{-1}$,
defined on some other probability space $(\Omega',\mcF',\bP')$ 
and such that $d(\sx_n(\omega), {\mathsf z}(\omega))\to 0$
for all $\omega\in \Omega'$.
By~\eqref{eq:expr-itp} and~\eqref{eq:rate-of-change-continuous}, the sequence 
of r.v.
$$
 \sx_n -  \sx_n(0) - \sG_{\gamma_n,R}(\sx_n)
$$
converges in probability to zero in $(\Omega',\mcF',\bP')$, as $n\to\infty$.
One can extract a subsequence under which this convergence holds in the almost sure sense.
Therefore, there exists an event of probability one s.t., everywhere on this event,  
$$
{\mathsf z}(t) = {\mathsf z}(0) 
 + \lim_{n\to\infty} \int_0^t\int_\Xi 
  h_{\gamma_n,R}(s, \sx_n(\gamma_n\lfloor u/\gamma_n\rfloor))\, \mu(ds) \, du\,
 \qquad (\forall t\geq 0) \,,
$$
where the limit is  taken along the former subsequence.
We now select an $\omega$ s.t. the above convergence holds, and omit the dependence on $\omega$ in the sequel
(otherwise stated, $\sz$ and $\sx_n$ are treated as elements of $C(\bR_+,E)$ and no longer as random variables).
Set $T>0$. As $(\sx_n)$ converges uniformly on $[0,T]$, there exists a compact set $K'$ (which depends on $\omega$) such that 
$\sx_n(\gamma_n\lfloor t/\gamma_n\rfloor)\in K'$ for all $t\in [0,T]$, $n\in \bN$.
Define
$$
v_n(s,t)\eqdef h_{\gamma_n,R}(s, \sx_n(\gamma_n\lfloor t/\gamma_n\rfloor))\,.
$$
\rev{By Eq.~\eqref{eq:moment-bis}, the sequence $(v_n,n\in \bN)$ 
forms a bounded subset of 
$\cL^{1+\epsilon_{K'}}\eqdef \cL^{1+\epsilon_{K'}}( \Xi\times [0,T], 
\mcG\otimes\mcB([0,T]), \mu\otimes\lambda_T; E)$.}
By the Banach-Alaoglu theorem, the sequence converges weakly to some mapping $v\in \cL^{1+\epsilon_{K'}}$ along some subsequence.
This has two consequences. First,
\begin{equation}
\ {\mathsf z}(t)={\mathsf z}(0)+
\int_0^t\int_\Xi v(s,u)\, \mu(ds)\, du\,,\quad(\forall t\in [0,T])\,.
\label{eq:di-int}
\end{equation}
Second, for $\mu\otimes\lambda_T$-almost all $(s,t)$, 
$v(s,t)\in H_R(s,{\mathsf z}(t))$.
In order to prove this point, remark that,  by Assumption~(RM),
$$
v_n(s,t) \to H_R(s,{\mathsf z}(t))
$$
for almost all $(s,t)$. This implies that the couple
$(\sx_n(\gamma_n\lfloor t/\gamma_n\rfloor),v_n(s,t))$
converges to $\mathrm{gr}(H_R(s, \cdot))$ and the second point thus follows from Lemma~\ref{lem:cvth}.
By  Fubini's theorem, there exists a negligible set of $[0,T]$ s.t.
for all $t$ outside this set, $v(\cdot,t)$ is an integrable selection of
$H_R(\cdot,{\mathsf z}(t))$.
As $H(\cdot,x)$ is integrable for every $x\in E$, the same holds for $H_R$.
Equation~(\ref{eq:di-int}) implies that
${\mathsf z}\in \Phi_{\sH_R}(K)\,.$
We have shown that for any sequence $((a_n,\gamma_n),n\in\bN)$  on 
$K\times (0,\gamma_0)$ s.t. $\gamma_n\to 0$,
there exists a subsequence along which, for every $\varepsilon>0$, $\bP^{a_n,\gamma_n}B_R^{-1}(d(\sX_{\gamma_n},\Phi_{\sH_R}(K))>\varepsilon)\to 0\,.$
This proves the lemma.
\end{proof}

\noindent {\bf End of the proof of Theorem~\ref{th:SA=wAPT}.}\\
We first prove the second statement.
Set an arbitrary $T>0$. Define $d_T(\sx,\sy)\eqdef \|x-y\|_{\infty,T}$.
It is sufficient to prove that for any sequence $((a_n,\gamma_n),n\in\bN)$ 
s.t.~$\gamma_n\to 0$, there exists a subsequence along which 
$\bP^{a_n,\gamma_n}(d_T(\sX_{\gamma_n},\Phi_{\sH}(K))>\varepsilon)\to 0$.
Choose $R>R_0(T)$, where $R_0(T)\eqdef\sup\{\|\sx(t)\|:t\in [0,T], \sx\in\Phi_\sH(a), a\in K\}$ is finite by
Assumption (RM).
It is easy to show that any $\sx\in\Phi_{\sH_R}(K)$ must satisfy $\|\sx \|_{\infty,T} < R$.
Thus, when $R>R_0(T)$, any $\sx\in \Phi_{\sH_R}(K)$ is such that there exists $\sy\in \Phi_{\sH}(K)$ with $d_T(\sx,\sy)=0$
\emph{i.e.}, the restrictions of $\sx$ and $\sy$ to $[0,T]$ coincide.
As a consequence of the Lemma~\ref{lem:aptTrunk}, each sequence $(a_n,\gamma_n)$ chosen as above admits 
a subsequence along which, for all $\varepsilon>0$, 
\begin{equation}
\bP^{a_n,\gamma_n}(d_T(\sX_{\gamma_n}\circ B_R,\Phi_{\sH}(K))>\varepsilon)\to 0\,.\label{eq:aptTronquee}
\end{equation}
The event $d_T(\sX_\gamma\circ B_R,\sX_\gamma)>0$ implies the event $\|\sX_\gamma\circ B_R\|_{\infty,T}\geq R$,
which in turn implies by the triangular inequality that
$
R_0(T) + d_T(\sX_\gamma\circ B_R,\Phi_{\sH}(K))\geq R\,.
$
Therefore,
\begin{equation}
\bP^{a_n,\gamma_n}( d_T(\sX_{\gamma_n}\circ B_R,\sX_{\gamma_n})>\varepsilon) \leq \bP( d_T(\sX_{\gamma_n}\circ B_R,\Phi_{\sH}(K))\geq R-R_0(T))\,.\label{eq:XTronq}
\end{equation}
By~(\ref{eq:aptTronquee}), the right hand side converges to zero. Using~\eqref{eq:aptTronquee} again along with the triangular
inequality, it follows that $\bP^{a_n,\gamma_n}(d_T(\sX_{\gamma_n},\Phi_{\sH}(K))>\varepsilon)\to 0$, which proves the second statement of the theorem.

We prove the first statement (tightness). By \cite[Theorem 7.3]{bil-(livre)99},
this is equivalent to showing that for every $T>0$, and for every sequence 
$(a_n,\gamma_n)$ on $K\times (0,\gamma_0)$, the sequence 
$( \bP^{a_n,\gamma_n} \sX_{\gamma_n}^{-1} p_0^{-1} )$ is tight, and 
for each positive $\varepsilon$ and $\eta$, there exists $\delta > 0$ such
that $\lim\sup_n 
\bP^{a_n,\gamma_n}\sX_{\gamma_n}^{-1}(\{x:w_x^T(\delta)>\varepsilon\}) < \eta$. 

First consider the case where $\gamma_n\to 0$. Fixing $T > 0$, letting 
$R>R_0(T)$ and using~\eqref{eq:XTronq}, it holds that for all $\varepsilon>0$, 
$\bP^{a_n,\gamma_n}( d_T(\sX_{\gamma_n}\circ B_R,\sX_{\gamma_n})>\varepsilon)
\to_n 0$. Moreover, we showed that 
$\bP^{a_n,\gamma_n}B_R^{-1}\sX_{\gamma_n}^{-1}$ is tight. 
The tightness of $( \bP^{a_n,\gamma_n} \sX_{\gamma_n}^{-1} p_0^{-1} )$ follows.
In addition, for every $\sx, \sy \in C(\bR_+,E)$, it holds by the triangle 
inequality that $w_{\sx}^T(\delta) \leq w_{\sy}^T(\delta) + 2 d_T(\sx, \sy)$
for every $\delta > 0$. Thus, 
\begin{multline*} 
\bP^{a_n,\gamma_n}\sX_{\gamma_n}^{-1}(\{x:w_x^T(\delta)>\varepsilon\})
\leq 
\bP^{a_n,\gamma_n}B_R^{-1}\sX_{\gamma_n}^{-1}
                               (\{x:w_x^T(\delta)>\varepsilon /2\}) \\ 
+ \bP^{a_n,\gamma_n}( d_T(\sX_{\gamma_n}\circ B_R,\sX_{\gamma_n})
  >\varepsilon / 4) , 
\end{multline*} 
which leads to the tightness of $(\bP^{a_n,\gamma_n} \sX_{\gamma_n}^{-1})$ 
when $\gamma_n \to 0$. 

It remains to establish the tightness when $\liminf_n\gamma_n>\eta>0$ for some
$\eta>0$. Note that for all $\gamma>\eta$,
$$
w_{\sX_{\gamma}^T(x)}(\delta)\leq 
 2\delta \max_{k=0\dots\lfloor T/\eta\rfloor+1}\|x_k\|\,.
$$
There exist $n_0$ such that for all $n\geq n_0$, $\gamma_n>\eta$ which implies 
by the union bound:
$$
\bP^{a_n,\gamma_n}\sX_{\gamma_n}^{-1}(\{\sx:w_\sx^T(\delta)>\varepsilon\}) \leq \sum_{k=0}^{\lfloor T/\eta\rfloor+1}P_\gamma^k(a,B(0,(2\delta)^{-1}\varepsilon)^c)\,,
$$
where $B(0,r)\subset E$ stands for the ball or radius $r$ and
where $P_\gamma^k$ stands for the iterated kernel, recursively defined by
\begin{equation}
P_\gamma^k(a,\cdot) = \int P_\gamma(a,dy)P_\gamma^{k-1}(y,\cdot)\label{eq:iterated}
\end{equation}
and $P_\gamma^0(a,\cdot)=\delta_a$. Using~\eqref{eq:moment}, it is an easy 
exercise to show, by induction, that for every $k\in \bN$, 
$P_\gamma^k(a,B(0,r)^c)\to 0$ as $r\to \infty$. By letting $\delta\to 0$ in the 
above inequality, the tightness of $(\bP^{a_n,\gamma_n}\sX_{\gamma_n}^{-1})$ 
follows.

\section{Proof of Proposition~\ref{prop:cluster}}
\label{sec-prf-cluster} 
To establish Prop.~\ref{prop:cluster}--\ref{cluster-Idroit}, we consider a 
sequence $((\pi_n,\gamma_n), n\in \bN)$ such that 
$\pi_n \in \cI(P_{\gamma_n})$, $\gamma_n \to 0$, and $(\pi_n)$ is tight. We 
first show that the sequence 
$(\upsilon_n \eqdef \bP^{\pi_n,\gamma_n} \sX_{\gamma_n}^{-1}, n\in \bN)$ is 
tight, then we show that every cluster point of $(\upsilon_n)$ satisfies 
the conditions of Def.~\ref{def-inv}. 

Given $\varepsilon > 0$, there exists a compact set $K\subset E$ such 
that $\inf_n \pi_n(K) > 1 - \varepsilon / 2$. By Th.~\ref{th:SA=wAPT}, 
the family 
$\{ \bP^{a,\gamma_n} \sX_{\gamma_n}^{-1} \, , \, a \in K, n \in \bN \}$ is
tight. Let $\cC$ be a compact set of $C(\bR_+, E)$ such that 
$\inf_{a\in K, n\in\bN} \bP^{a,\gamma_n} \sX_{\gamma_n}^{-1}(\cC) > 
1 - \varepsilon/2$. 
By construction of the probability measure $\bP^{\pi_n, \gamma_n}$, it 
holds that 
$\bP^{\pi_n, \gamma_n}(\cdot) = \int_E \bP^{a, \gamma_n}(\cdot) \, \pi_n(da)$. 
Thus, 
\[
\upsilon_n(\cC) 
\geq \int_K \bP^{a, \gamma_n}\sX_{\gamma_n}^{-1}(\cC) \, \pi_n(da) 
 > (1 - \varepsilon/2)^2 > 1 - \varepsilon \, , 
\]
which shows that $(\upsilon_n)$ is tight. 

Since $\pi_n = \upsilon_n p_0^{-1}$, and since the projection $p_0$ is
continuous, it is clear that every cluster point $\pi$ of $\cI(\cP)$ as 
$\gamma\to 0$ can be written as $\pi = \upsilon p_0^{-1}$, where $\upsilon$ is 
a cluster point of a sequence $(\upsilon_n)$. Thus,   
Def.~\ref{def-inv}--\ref{inv-margin} is satisfied by $\pi$ and $\upsilon$. 
To establish Prop.~\ref{prop:cluster}--\ref{cluster-Idroit}, we need to verify
the conditions~\ref{inv-support} and~\ref{inv-Theta} of 
Definition~\ref{def-inv}. In the remainder of the proof, we denote with a 
small abuse as $(n)$ a subsequence along which $(\upsilon_n)$ converges 
narrowly to $\upsilon$. 

To establish the validity of 
Def.~\ref{def-inv}--\ref{inv-support}, we prove that for every $\eta > 0$, 
$\upsilon_n( (\Phi_\sH(E))_\eta ) \to 1$ as $n\to\infty$; the result will
follow from the convergence of $(\upsilon_n)$.  Fix $\varepsilon > 0$, and let
$K \subset E$ be a compact set such that 
$\inf_n \pi_n(K) > 1 - \varepsilon$.  We have 
\begin{align*}
\upsilon_n( (\Phi_\sH(E))_\eta ) &= 
\bP^{\pi_n,\gamma_n}( d(\sX_{\gamma_n}, \Phi_{\sH}(E)) < \eta ) \\ 
&\geq \bP^{\pi_n,\gamma_n}( d(\sX_{\gamma_n}, \Phi_{\sH}(K)) < \eta ) \\
&\geq \int_K \bP^{a,\gamma_n}( d(\sX_{\gamma_n}, \Phi_{\sH}(K)) < \eta ) 
    \, \pi_n(da)  \\ 
&\geq (1-\varepsilon) 
\inf_{a\in K} \bP^{a,\gamma_n}( d(\sX_{\gamma_n}, \Phi_{\sH}(K)) < \eta ) 
\, .
\end{align*} 
By Th.~\ref{th:SA=wAPT}, the infimum at the right hand side converges to $1$.
Since $\varepsilon > 0$ is arbitrary, we obtain the result. 

It remains to establish the $\Theta$-invariance of $\upsilon$ 
(Condition~\ref{inv-Theta}). Equivalently, we need to show that 
\begin{equation} 
\int f(\sx) \, \upsilon(d\sx) = \int f \circ \Theta_t (\sx) \, 
\upsilon(d\sx) 
\label{inv-ups} 
\end{equation} 
for all $f\in C_b(C(\bR_+, E))$ and all $t > 0$.
We shall work on $(\upsilon_n)$ and make $n\to\infty$. Write 
$\eta_n \eqdef t - \gamma_n \lfloor t / \gamma_n \rfloor$. 
Thanks to the $P_{\gamma_n}$--invariance of $\pi_n$,
$\Theta_{\eta_n}(\sx(\gamma_n \lfloor t / \gamma_n \rfloor+\cdot))$ and
$\Theta_{t} (\sx)$ are equal in law under $\upsilon_n(d\sx)$. Thus, 
\begin{align} 
\int f \circ \Theta_t (\sx) \, \upsilon_n(d\sx) &= 
\int 
f \circ \Theta_{\eta_n} (\sx(\gamma_n \lfloor t / \gamma_n \rfloor+\cdot)) 
  \, \upsilon_n(d\sx) \nonumber \\  
&= \int f \circ \Theta_{\eta_n} (\sx) \, \upsilon_n(d\sx) . \label{eq:commentonlappelle}
\end{align} 
Using Skorokhod's representation theorem, there exists a probability
space $(\Omega',\mathcal F', \bP')$ and random variables
$(\bar \sx_n, n\in \bN)$ and $\bar \sx$ over this probability space, with
values in $C(\bR_+, E)$, such that for every $n \in \bN$, the
distribution of $\bar \sx_n$ is $\upsilon_n$, the
distribution of $\bar \sx$ is $\upsilon$ and
$\bP'$-a.s, 
\[
d(\bar \sx_n,\bar \sx) \longrightarrow_{n \to +\infty} 0,
\]
\textit{i.e}, $(\bar\sx_n)$ converges to $\bar \sx$ as
$n \to +\infty$ uniformly over compact sets of $\bR_+$. Since
$\eta_n \rightarrow_{n \to +\infty} 0$, $\bP'$-a.s,
$d(\Theta_{\eta_n}(\bar\sx_n),\bar\sx) \longrightarrow_{n \to
  +\infty} 0.$ Hence,
\[
\int f \circ \Theta_{\eta_n} (\sx) \, \upsilon_n(d\sx)
\xrightarrow[n\to\infty]{} \int f(\sx) \, \upsilon(d\sx) \,.
\]
Recalling Eq.~(\ref{eq:commentonlappelle}), we have shown that $\int f\circ \Theta_t (\sx) \, \upsilon_n(d\sx) \xrightarrow[n\to\infty]{} 
\int f\circ \Theta_t(\sx) \, \upsilon(d\sx)$. Since $\int f(\sx) \, \upsilon_n(d\sx) \xrightarrow[n\to\infty]{} 
\int f(\sx) \, \upsilon(d\sx)$, the identity \eqref{inv-ups} holds true. 
Prop.~\ref{prop:cluster}--\ref{cluster-Idroit} is proven. 

We now prove Prop.~\ref{prop:cluster}--\ref{cluster-Icursif}. Consider a 
sequence $((\mathfrak m_n, \gamma_n), n\in \bN)$ such that 
$\mathfrak m_n \in \mcI(P_{\gamma_n})$, $\gamma_n \to 0$, and 
$\mathfrak m_n \Rightarrow \mathfrak m$ for some $\mathfrak m \in\cM(\cM(E))$. 
Since the space $\cM(E)$ is separable, Skorokhod's representation theorem shows
that there exists a probability space $(\Omega',\mcF',\bP')$, a 
sequence of $\Omega' \to \cM(E)$ random variables $(\Lambda_n)$ with 
distributions $\mathfrak m_n$, and a $\Omega' \to \cM(E)$ random variable 
$\Lambda$ with distribution $\mathfrak m$ such that $\Lambda_n(\omega) 
\Rightarrow \Lambda(\omega)$ for each $\omega \in \Omega'$. Moreover, there is a 
probability one subset of $\Omega'$ such that $\Lambda_n(\omega)$ is a 
$P_{\gamma_n}$--invariant probability measure for all $n$ and for every 
$\omega$ in this set. For each of these $\omega$, we can construct on the 
space $(E^\bN, \mcF)$ a probability measure 
$\bP^{\Lambda_n(\omega), \gamma_n}$ as we did in Sec.~\ref{rm-general}. By the 
same argument as in the proof of 
Prop.~\ref{prop:cluster}--\ref{cluster-Idroit}, the sequence 
$(\bP^{\Lambda_n(\omega), \gamma_n} \sX_{\gamma_n}^{-1}, n\in\bN)$ is tight, 
and any cluster point $\upsilon$ satisfies the conditions of 
Def.~\ref{def-inv} with $\Lambda(\omega) = \upsilon p_0^{-1}$. 
Prop.~\ref{prop:cluster} is proven. 

\section{Proof of Theorem~\ref{the:CV}}
\label{sec-prf-CV} 

\subsection{Technical lemmas}
\label{sec:technical-lemmas}

\begin{lemma}
  \label{lem:cpctLevy}
Given a family $\{ K_j , j\in\bN \}$ of compact sets of $E$, the
set
\[
U\eqdef\{ \nu \in \cM(E) \, : \, \forall j \in \bN, 
\nu(K_j) \geq 1 - 2^{-j} \}
\]
is a compact set of $\cM(E)$.
\end{lemma}
\begin{proof}
The set $U$ is tight hence relatively compact by Prokhorov's theorem. It is moreover
closed. Indeed, let $(\nu_n, n\in\bN)$ represent a sequence of $U$ s.t. 
$\nu_n\Rightarrow \nu$. Then, for all $j\in \bN$, 
$\nu(K_j) \geq \limsup_n \nu_n(K_j) \geq 1 - 2^{-j}$ since
$K_j$ is closed. 
\end{proof}
\begin{lemma}
\label{lem:01}
Let $X$ be a real random variable such that $X \leq 1$ with probability one,
and $\bE X \geq 1 - \varepsilon$ for some $\varepsilon \geq 0$. Then
$\bP[ X \geq 1 - \sqrt{\varepsilon}] \geq 1 - \sqrt{\varepsilon}$. 
\end{lemma}
\begin{proof}
$1-\varepsilon \leq \bE X \leq \bE X \1_{X < 1-\sqrt{\varepsilon}} 
+ \bE X \1_{X \geq 1-\sqrt{\varepsilon}} \leq 
(1-\sqrt{\varepsilon})(1-\bP[X\geq 1 - \sqrt{\varepsilon}]) + 
\bP[X\geq 1 - \sqrt{\varepsilon}]$. The result is obtained by rearranging. 
\end{proof}

For any $\mathfrak{m}\in\cM(\cM(E))$, we denote by $e(\mathfrak{m})$ the probability measure in $\cM(E)$ such that
for every $f\in C_b(E)$,
$$
e(\mathfrak{m}) :f\mapsto \int\nu(f)\mathfrak{m}(d\nu)\,.
$$
Otherwise stated, $e(\mathfrak{m})(f) = \mathfrak{m}(\cT_f)$ where $\cT_f:\nu\mapsto \nu(f)$. 
\begin{lemma}
\label{lem:espTight}
  Let $\cL$ be a family on $\cM(\cM(E))$. If
$\{e(\mathfrak m)\,:\,\mathfrak m\in\cL\}$ is tight, then $\cL$ is tight.
\end{lemma}
\begin{proof}
Let $\varepsilon>0$ and choose any integer $k$ s.t. $2^{-k+1}\leq\varepsilon$.
For all $j\in\bN$, choose a compact set $K_j\subset E$ s.t. for all $\mathfrak m\in\cL$,
$e(\mathfrak m)(K_j) > 1-2^{-2j}\,.$ Define $U$ as the set of measures $\nu\in\cM(E)$
s.t. for all $j\geq k$, $\nu(K_j)\geq 1-2^{-j}$. By Lemma~\ref{lem:cpctLevy}, $U$ is compact.
For all $\mathfrak m\in\cL$, the union bound implies that 
\begin{align*}
  \mathfrak m(E\backslash U) &\leq \sum_{j=k}^\infty \mathfrak m\{\nu:\nu(K_j)<1-2^{-j}\}
\end{align*}
By Lemma~\ref{lem:01}, $\mathfrak m\{\nu:\nu(K_j)\geq 1-2^{-j}\}\geq  1- 2^{-j}$. Therefore,
$
\mathfrak m(E\backslash U) \leq \sum_{j=k}^\infty 2^{-j} = 2^{-k+1}\leq \varepsilon\,.
$
This proves that $\cL$ is tight.
\end{proof}

\begin{lemma} 
\label{lem:CVe}
Let $(\mathfrak m_n,n\in \bN)$ be a sequence on $\cM(\cM(E))$,
and consider $\bar{\mathfrak{m}}\in \cM(\cM(E))$.  
If $\mathfrak{m}_n\Rightarrow \bar{\mathfrak m}$, then
  $e(\mathfrak{m}_n)\Rightarrow e(\bar{\mathfrak m})$.
\end{lemma}
\begin{proof}
  For any $f\in C_b(E)$, $\cT_f\in C_b(\cM(E))$. Thus,
$\mathfrak{m}_n(\cT_f)\to \bar{\mathfrak m}(\cT_f)$.
\end{proof}

When a sequence $(\mm_n, n\in\bN)$ of $\cM(\cM(E))$ converges narrowly to
$\mm\in \cM(\cM(E))$, it follows from the above proof that $\mm_n\cT_f^{-1}
\Rightarrow \mm\cT_f^{-1}$ for all bounded continuous $f$. The purpose of the
next lemma is to extend this result to the case where $f$ is not necessarily
bounded, but instead, satisfies some uniform integrability condition.  For any
vector-valued function $f$, we use the notation  $\|f\|\eqdef\|f(\cdot)\|$.

\begin{lemma}
  \label{lem:UIf}
Let $f\in C(E,\bR^{N'})$ where $N'\geq 1$ is an integer.
Define by $\cT_f:\cM(E)\to\bR$ the mapping s.t. $\cT_f(\nu) \eqdef \nu(f)$ if 
$\nu(\|f\|)<\infty$ and equal to zero otherwise.
Let $(\mm_n,n\in \bN)$ be a sequence on $\cM(\cM(E))$ and let $\mm\in\cM(\cM(E))$. Assume that $\mm_n\Rightarrow \mm$
and 
\begin{equation}
\lim_{K\to\infty}\sup_n e(\mm_n)(\|f\|\1_{\|f\|>K})=0\,.\label{eq:UI1}
\end{equation}
Then, $\nu(\|f\|)<\infty$ for all $\nu$ $\mm$-a.e. and $\mm_n\cT_f^{-1}\Rightarrow \mm\cT_f^{-1}$.
\end{lemma}
\begin{proof}
By Eq.~(\ref{eq:UI1}),  $e(\mm)(\|f\|)<\infty$.
This implies that for all $\nu$ $\mm$-a.e., $\nu(\|f\|)<\infty$.
Choose $h\in C_b(\bR^{N'})$ s.t. $h$ is $L$-Lipschitz continuous. 
We must prove that $\mm_n\cT_f^{-1}(h)\to \mm\cT_f^{-1}(h)$.
By the above remark, $\mm\cT_f^{-1}(h) = \int h(\nu(f))d\mm(\nu)$, and by Eq~(\ref{eq:UI1}),
$\mm_n\cT_f^{-1}(h) = \int h(\nu(f))d\mm_n(\nu)$. Choose $\varepsilon>0$. 
By Eq.~(\ref{eq:UI1}), there exists $K_0>0$ s.t. for all $K>K_0$, 
$\sup_ne(\mm_n)(\|f\|\1_{\|f\|>K})<\varepsilon$.
For every such $K$, define the bounded function $f_K\in C( E,\bR^{N'})$ by $f_K(x) = f(x) (1\wedge K/\|f(x)\|)$. 
For all $K>K_0$, and for all $n\in \bN$,
\begin{align*}
|\mm_n\cT_f^{-1}(h) - \mm_n\cT_{f_K}^{-1}(h)| & \leq  \int |h(\nu(f))-h(\nu(f_K))|d\mm_n(\nu)\\
&\leq L\,\int \nu(\|f-f_K\|)d\mm_n(\nu)\\
&\leq L\,\int \nu(\|f\|\1_{\|f\|>K})d\mm_n(\nu)\leq L\varepsilon\,.
\end{align*}
By continuity of $\cT_{f_K}$, it holds that $ \mm_n\cT_{f_K}^{-1}(h)\to \mm\cT_{f_K}^{-1}(h)$.
Therefore, for every $K>K_0$, $\limsup_n  |\mm_n\cT_f^{-1}(h) - \mm\cT_{f_K}^{-1}(h)| 
\leq  L\varepsilon\,.$
As $\nu(\|f\|)<\infty$ for all $\nu$ $\mm$-a.e., the dominated convergence theorem implies that $\nu(f_K)\to\nu(f)$ as $K\to\infty$,
$\mm$-a.e. As $h$ is bounded and continuous, a second application of the dominated convergence theorem
implies that $\int h(\nu(f_K))d\mm(\nu)\to\int h(\nu(f))d\mm(\nu)$, which reads $ \mm\cT_{f_K}^{-1}(h)\to  \mm\cT_{f}^{-1}(h)$.
Thus, $\limsup_n  |\mm_n\cT_{f}^{-1}(h) -  \mm\cT_{f}^{-1}(h)| \leq L\varepsilon\,.$
As a consequence, $\mm_n\cT_{f}^{-1}(h) \to \mm\cT_{f}^{-1}(h)$ as $n\to\infty$, which completes the proof.
\end{proof}

\subsection{Narrow Cluster Points of the Empirical Measures}

Let $P:E\times \mcB(E)\to [0,1]$ be a probability transition kernel. 
For $\nu\in \cM(E)$, we denote by 
 $\bP^{\nu,P}$ the probability on $(\Omega,\mcF)$
such that $X$ is an homogeneous Markov chain with initial distribution $\nu$ and transition kernel $P$.

For every $n\in\bN$, we define the measurable mapping $\Lambda_n:\Omega\to\cM(E)$ as
\begin{equation}
\Lambda_n(x) \eqdef \frac 1{n+1}\sum_{k=0}^n\delta_{x_k}\label{eq:Lambdan}
\end{equation}
for all $x=(x_k:k\in\bN)$. Note that
$$
\bE^{\nu,P}\Lambda_n = \frac 1{n+1} \sum_{k=0}^{n}\nu P^k\,,
$$
where $\bE^{\nu,P}\Lambda_n = e(\bP^{\nu,P}\Lambda_n^{-1})$, and $P^k$ stands for the iterated kernel, recursively defined by
$
P^k(x,\cdot) = \int P(x,dy)P^{k-1}(y,\cdot)
$ and $P^0(x,\cdot)=\delta_x$.

We recall that $\mcI(P)$ represents the subset of $\cM(\cM(E))$ formed by the measures
whose support is included in $\cI(P)$.
\begin{proposition}
\label{prop:feller}
Let $P:E\times \mcB(E)\to [0,1]$ be a Feller probability transition kernel. Let $\nu\in\cM(E)$. 
\begin{enumerate}
\item Any cluster point of $\{\bE^{\nu,P}\Lambda_n \,,\,n\in\bN \}$ is an element of 
$\cI(P)$.
\item Any cluster point of 
$\{\bP^{\nu,P}\Lambda_n^{-1}\,,\,n\in\bN \}$ is an element of $\mcI(P)$.
\end{enumerate}
\end{proposition}

\begin{proof}
We omit the upper script $^{\nu,P}$.
For all $f\in C_b(E)$,
$\bE\Lambda_n(Pf) -\bE\Lambda_n(f) \to 0$. As $P$ is Feller, any cluster point  $\pi$  of
$\{\bE\Lambda_n\,,\,n\in\bN\}$ satisfies $\pi(P f)=\pi(f)$.
This proves the first point.

For every $f\in C_b(E)$ and $x\in \Omega$, consider the decomposition:
\begin{align*}
  \Lambda_{n}(x)(P f) - \Lambda_{n}(x)( f)
&= \frac  {1}{n+1}\sum_{k=0}^{n-1}(Pf(x_k)-f(x_{k+1}))+\frac{Pf(x_n)-f(x_0)}{n+1}\,.
\end{align*}
Using that $f$ is bounded, Doob's martingale convergence theorem implies that the sequence
$
\Bigl( \sum_{k=0}^{n-1}k^{-1}(Pf(X_k)-f(X_{k+1})) \Bigr)_n 
$
converges a.s. when $n$ tends to infinity. By Kronecker's lemma, we deduce that 
$\frac  {1}{n+1}\sum_{k=0}^{n-1}(Pf(X_k)-f(X_{k+1}))$ tends a.s. to zero. Hence,
\begin{equation}
\label{eq:mtg}
\Lambda_{n}(P f) - \Lambda_{n}( f) \to 0\ \text{a.s.}
\end{equation}
Now consider a subsequence 
$(\Lambda_{\varphi_n})$ which converges in distribution to some
r.v. $\Lambda$ as $n$ tends to infinity. 
For a fixed $f\in C_b(E)$, the mapping
$\nu \mapsto (\nu(f),\nu(Pf))$ on $\cM(\bR) \to \bR^2$ is continuous. 
From the mapping theorem,  $\Lambda_{\varphi_n}(f)-\Lambda_{\varphi_n}(Pf)$ converges in distribution to
$\Lambda(f)-\Lambda(Pf)$. By~(\ref{eq:mtg}), it follows that  $\Lambda(f)-\Lambda(Pf)=0$
on some event $\cE_f\in \mcF$ of probability one.
Denote by $C_\kappa(E)\subset C_b(E)$ the set of continuous real-valued functions having a compact support,
and let $C_\kappa(E)$ be equipped with the uniform norm $\|\cdot\|_\infty$.
Introduce a dense denumerable subset $S$ of $C_\kappa(E)$. On the 
probability-one event $\cE=\cap_{f\in S}\cE_f$, it holds that for all $f\in S$,
$\Lambda(f)=\Lambda P(f)$.
Now consider $g\in C_\kappa(E)$ and let $\varepsilon>0$. Choose $f\in S$ such that $\|f-g\|_\infty\leq \varepsilon$.
Then, almost everywhere on $\cE$, $|\Lambda(g) -\Lambda P(g)| \leq |\Lambda(f)-\Lambda(g)| + |\Lambda P(f)-\Lambda P(g)|\leq 2\varepsilon$.
Thus, $\Lambda(g) -\Lambda P(g) =0$ for every $g\in C_\kappa(E)$. 
Hence, almost everywhere on $\cE$, one has $\Lambda =\Lambda P$.
\end{proof}

\subsection{Tightness of the Empirical Measures}

\begin{proposition}
\label{prop:tight}
Let $\cP$ be a family of transition kernels on $E$. Let $V:E\to[0,+\infty)$, 
$\psi:E\to[0,+\infty)$ be measurable. Let $\alpha:\cP\to(0,+\infty)$ and
$\beta:\cP\to\bR$.
Assume that $\sup_{P\in\cP}\frac{\beta(P)}{\alpha(P)}<\infty$ and 
$\psi(x) \to \infty$ as $\| x \| \to\infty$. Assume that for every $P\in\cP$,
$$
P V\leq V-\alpha(P)\psi+\beta(P)\,.
$$
Then, the following holds.
\begin{enumerate}[i)]
\item \label{it:itight} The family $\bigcup_{P\in\cP}\cI(P)$ is tight. Moreover, $\sup_{\pi\in\cI(\cP)} \pi(\psi) < +\infty\,.$

\item \label{it:mtight} For every $\nu\in\cM(E)$ s.t. $\nu(V)<\infty$, every $P\in \cP$, 
$\{\bE^{\nu,P}\Lambda_n\,,\,n\in\bN \}$ is tight. Moreover, $\sup_{n\in\bN}\bE^{\nu,P}\Lambda_n(\psi)<\infty\,.$
\end{enumerate} 
\end{proposition}
\begin{proof}
  For each $P\in\cP$, $P V$ is everywhere finite by assumption. Moreover,
$$
\sum_{k=0}^n P^{k+1}V \leq \sum_{k=0}^n P^{k}V -\alpha(P)\sum_{k=0}^n P^k \psi +(n+1) \beta(P)\,.
$$ 
Using that $V\geq 0$ and $\alpha(P)>0$, 
$$
\frac{1}{n+1}\sum_{k=0}^n P^k \psi \leq \frac{V}{\alpha(P)(n+1)}  +c\,,
$$
where $c\eqdef\sup_{P\in\cP}\beta(P)/\alpha(P)$ is finite. For any $M>0$,
\begin{align}
  \frac{1}{n+1}\sum_{k=0}^n P^k (\psi\wedge M) &\leq
  \left(\frac{1}{n+1}\sum_{k=0}^n P^k \psi\right)\wedge M \nonumber \\ &
\leq \left(\frac{V}{\alpha(P)(n+1)} +c\right)\wedge M\,. \label{eq:wedge}
\end{align}
Set $\pi\in\cI(\cP)$, and consider
$P\in\cP$ such that $\pi=\pi P$. Inequality~(\ref{eq:wedge}) implies that
for every $n$,
$$
\pi (\psi\wedge M) \leq
  \pi\left(\left(\frac{V}{\alpha(P)(n+1)} +c\right)\wedge M\right)\,.
$$
By Lebesgue's dominated convergence theorem, $\pi (\psi\wedge M) \leq c$.
Letting $M\to\infty$ yields $\pi(\psi)\leq c$. The tightness of $\cI(\cP) $ 
follows from the convergence of $\psi(x)$ to $\infty$ as $\|x\|\to\infty$. 
Setting $M=+\infty$ in~(\ref{eq:wedge}), and integrating w.r.t. $\nu$, we obtain
$$
  \bE^{\nu,P}\Lambda_n(\psi)  \leq \frac{\nu(V)}{(n+1)\alpha(P)} +c\,,
$$ 
which proves the second point.
\end{proof}

\begin{proposition}
\label{prop:tight2}
We posit the assumptions of Prop.~\ref{prop:tight}. Then,
\begin{enumerate}
\item The family $\mcI(\cP)\eqdef \bigcup_{P\in\cP}\mcI(P)$ is tight;
\item $\{\bP^{\nu,P}\Lambda_n^{-1}\,,\,n\in\bN\}$ is tight.
\end{enumerate} 
\end{proposition}

\begin{proof}
For every $\mathfrak{m}\in\mcI(\cP)$, it is easy to see that
$e(\mathfrak{m})\in\cI(\cP)$. Thus, 
$\{e(\mathfrak{m}):\mathfrak{m}\in \mcI(\cP)\}$ is tight by 
Prop.~\ref{prop:tight}. By Lemma~\ref{lem:espTight}, $\mcI(\cP)$ is tight.
The second point follows from the equality 
$\bE^{\nu,P}\Lambda_n=e(\bP^{\nu,P}\Lambda_n^{-1})$ along with Prop.~\ref{prop:tight} and 
Lemma~\ref{lem:espTight}.
\end{proof}

\subsection{Main Proof}

By continuity of $h_\gamma(s,\cdot)$ for every $s\in \Xi$, $\gamma\in (0,\gamma_0)$, the transition kernel
$P_\gamma$ is Feller. 
By Prop.~\ref{prop:tight} and Eq.~\eqref{eq:dVP}, we have
$\sup_n \bE^{\nu,\gamma}\Lambda_n(\varphi\circ f)<\infty$ which, by de la Vall\'ee-Poussin's criterion for uniform integrability, implies
\begin{equation}
\lim_{K\to\infty}\sup_n \bE^{\nu,\gamma}\Lambda_n(\|f\|\1_{\|f\|>K})=0\,.\label{eq:UI-ELambda}
\end{equation}
In particular, the quantity $\bE^{\nu,\gamma}\Lambda_n(f)=\bE^{\nu,\gamma}(F_n)$ is well-defined.

We now prove the statement (\ref{eq:CVSE}).
By contradiction, assume that for some $\delta>0$, there exists a positive sequence $\gamma_j\to 0$, such that for all $j\in\bN$,
$
\limsup_{n\to\infty}\  d\left(\bE^{\nu,\gamma_j}\Lambda_n(f)\,,\cS_f\right)>\delta\,.
$
For every $j$, there exists an increasing sequence of integers 
$(\varphi_n^j,n\in\bN)$ converging to $+\infty$ s.t. 
\begin{equation}
\forall n,\ d\left(\bE^{\nu,\gamma_j}\Lambda_{\varphi_n^j}(f)\,,\cS_f\right)>\delta\,.\label{eq:contradiction1}
\end{equation}
By Prop.~\ref{prop:tight}, the sequence 
$(\bE^{\nu,\gamma_j}\Lambda_{\varphi_n^j},n\in\bN)$ is tight. By Prokhorov's
theorem and Prop.~\ref{prop:feller}, there exists $\pi_j\in \cI(P_{\gamma_j})$
such that, as $n$ tends to infinity, $\bE^{\nu,\gamma_j}\Lambda_{\varphi_n^j}\Rightarrow \pi_j$ along some subsequence.
By the uniform integrability condition~(\ref{eq:UI-ELambda}), $\pi_j(\|f\|)<\infty$ and
$\bE^{\nu,\gamma_j}\Lambda_{\varphi_n^j}(f)\to \pi_j(f)$ as~$n$ tends to infinity, along the latter subsequence.
By Eq. (\ref{eq:contradiction1}), for all $j\in \bN$, 
$d(\pi_j(f),\cS_f)\geq\delta\,.$
By Prop.~\ref{prop:tight}, $\sup_{\pi\in\cI(\cP)} \pi(\psi) < +\infty\,.$ Since $\varphi\circ f\leq M(1+\psi)$,
de la Vall\'ee-Poussin's criterion again implies that
\begin{equation}
  \label{eq:UI-j}
  \sup_{\pi\in \cI(\cP)} \pi(\|f\|\1_{\|f\|>K}) <\infty\,.
\end{equation}
Also by Prop.~\ref{prop:tight}, the sequence $(\pi_j)$ is tight. 
Thus $\pi_j\Rightarrow \pi$ along some subsequence, for some measure $\pi$ which, by Prop.~\ref{prop:cluster}, 
is invariant for $\Phi_{\sH}$.
The uniform integrability condition~\eqref{eq:UI-j} implies that 
$\pi(\|f\|)<\infty$ (hence, the set $\cS_f$ is non-empty)
and  $\pi_j(f)\to \pi(f)$ as $j$ tends to infinity, along the above subsequence. 
This shows that $d(\pi(f),\cS_f)>\delta$, which is absurd.
The statement (\ref{eq:CVSE}) holds true (and in particular, $\cS_f$ must be non-empty).

The proof of the statement~(\ref{eq:support}) follows the same line, by replacing $f$ with the
function $\1_{\overline{\mU_\epsilon}}$. We briefly explain how the proof adapts, without repeating all the arguments.
In this case, $\cS_{\1_{{\mU_\epsilon^c}}}$ is the singleton $\{0\}$, and Equation~(\ref{eq:contradiction1})
reads $\bE^{\nu,\gamma_j}\Lambda_{\varphi_n^j}({\mU_\epsilon^c})>\delta$.
By the Portmanteau theorem,
 $\limsup_n\bE^{\nu,\gamma_j}\Lambda_{\varphi_n^j}({\mU_\epsilon^c})\leq \pi_j({\mU_\epsilon^c})$
where the $\limsup$ is taken along some subsequence.
The contradiction follows from the fact that $\limsup \pi_j({\mU_\epsilon^c})\leq \pi({\overline{\mU_\epsilon^c}})=0$
(where the $\limsup$ is again taken along the relevant subsequence).

We prove the statement~(\ref{eq:CVSf}). Assume by contradiction that for some (other) sequence $\gamma_j\to 0$,
$\limsup_{n\to\infty}\  \bP^{\nu,\gamma_j}\left(d\left(\Lambda_{n}(f)\,,\cS_f\right)\geq \varepsilon\right)>\delta\,.$
For every $j$, there exists a sequence $(\varphi_n^j,n\in \bN)$ s.t. 
\begin{gather}
\forall n,\ \bP^{\nu,\gamma_j}\left(d\left(\Lambda_{\varphi_n^j}(f)\,,\cS_f\right)\geq \varepsilon\right)>\delta\,.
\label{eq:contradition2}
\end{gather}
By Prop.~\ref{prop:tight2}, $(\bP^{\nu,\gamma_j}\Lambda_{\varphi_n^j}^{-1},n\in\bN)$ is tight, one can extract a further subsequence
(which we still denote by $(\varphi_n^j)$ for simplicity) s.t. $\bP^{\nu,\gamma_j}\Lambda_{\varphi_n^j}^{-1}$ converges narrowly 
to a measure ${\mathfrak m}_j$ as $n$ tends to infinity, which, by Prop.~\ref{prop:feller}, satisfies $\mm_j\in \mcI(P_{\gamma_j})$.
Noting that $e(\bP^{\nu,\gamma_j}\Lambda_{\varphi_n^j}^{-1})=\bE^{\nu,\gamma_j}\Lambda_{\varphi_n^j}$ and recalling Eq.~(\ref{eq:UI-ELambda}),
Lemma~\ref{lem:UIf} implies that $\nu'(\|f\|)<\infty$ for all $\nu'$ $\mm_j$-a.e., and
$\bP^{\nu,\gamma_j}\Lambda_{\varphi_n^j}^{-1}\cT_f^{-1}\Rightarrow \mm_j\cT_f^{-1}$, where we recall that
$\cT_f(\nu') \eqdef \nu'(f)$ for all $\nu'$ s.t.  $\nu'(\|f\|)<\infty$. As $(\cS_f)_\varepsilon^c$ is a closed set, 
\begin{align*}
  \mm_j\cT_f^{-1}((\cS_f)_\varepsilon^c) &\geq \limsup_n \bP^{\nu,\gamma_j}\Lambda_{\varphi_n^j}^{-1}\cT_f^{-1}((\cS_f)_\varepsilon^c) \\
&= \limsup_n \bP^{\nu,\gamma_j}\left(d\left(\Lambda_{\varphi_n^j}(f)\,,\cS_f\right)\geq \varepsilon\right)>\delta\,.
\end{align*}
By Prop.~\ref{prop:tight}, $(\mm_j)$ is tight, and one can extract a subsequence (still denoted by $(\mm_j)$)
along which $\mm_j\Rightarrow \mm$ for some measure $\mm$ which, by Prop.~\ref{prop:cluster}, belongs to $\mcI(\Phi_{\sH})$.
For every $j$, $e(\mm_j)\in \cI(P_{\gamma_j})$. By the uniform integrability condition (\ref{eq:UI-j}), one can 
apply Lemma~\ref{lem:UIf} to the sequence $(\mm_j)$. We deduce that $\nu'(\|f\|)<\infty$ for all $\nu'$ $\mm$-a.e.
and $\mm_j\cT_f^{-1}\Rightarrow \mm\cT_f^{-1}$. In particular, 
$$
\mm\cT_f^{-1}((\cS_f)_\varepsilon^c)\geq \limsup_j \mm_j\cT_f^{-1}((\cS_f)_\varepsilon^c) >\delta\,.
$$
Since $\mm\in \mcI(\Phi_{\sH})$, it holds that $\mm\cT_f^{-1}((\cS_f)_\varepsilon^c)=0$, hence a contradiction.

\section{Proofs of Theorems~\ref{cvg-CVSI} and~\ref{cvg-XY}}
\label{sec-prf-asymptotics} 

\subsection{Proof of Theorem~\ref{cvg-CVSI}}

In this proof, we set $L=L_{\aver(\Phi)}$ to simplify the notations.
It is straightforward to show that the identity mapping $f(x)=x$ satisfies the
hypotheses of Th.~\ref{the:CV} with $\varphi = \psi$. Hence, it is sufficient to prove that
$\cS_f$ is a subset of $\overline{\co}(L)$, the closed convex hull of $L$. Choose $q\in \cS_I$ and let
$q=\int xd\pi(x)$ for some $\pi\in\cI(\Phi)$ admitting a first order
moment. There exists a $\Theta$-invariant measure 
$\upsilon\in \cM(C(\bR_+,E))$ s.t. $\support(\upsilon)\subset\Phi(E)$ and $\upsilon p_0^{-1}=\pi$.
We remark that for all $t>0$,
\begin{equation}
q = \upsilon(p_0) = \upsilon(p_t) = \upsilon(p_{t}\circ \aver)\,,\label{eq:pnu}
\end{equation}
where the second identity is due to the shift-invariance of $\upsilon$, and the last one uses
Fubini's theorem.
Again by the shift-invariance of $\upsilon$, the family $\{p_t,t>0\}$ is uniformly integrable w.r.t. $\upsilon$.
By Tonelli's theorem, $\sup_{t>0}\upsilon(\|p_t\circ\aver \|\1_S)\leq  \sup_{t>0}  \upsilon(\|p_t\|\1_S)$
for every $S\in\mcB(C(\bR_+,E))$. Hence, the family $\{p_t\circ \aver,t>0\}$ is $\upsilon$-uniformly integrable as well.
In particular, $\{p_t\circ \aver,t>0\}$ is tight in
$(C(\bR_+,E),\mcB(C(\bR_+,E)),\upsilon)$.  
By Prokhorov's theorem, there exists
a sequence $t_n\to\infty$ and a measurable function 
$g:C(\bR_+,E)\to E$ such that $p_{t_n}\circ \aver$ converges in distribution to $g$
as $n\to\infty$. By uniform integrability, $\upsilon(p_{t_n}\circ \aver)\to \upsilon(g)$.
Equation~(\ref{eq:pnu}) finally implies that 
$$
q=\upsilon(g)\,.
$$
In order to complete the proof, it is sufficient to show that $g(\sx)\in \overline L$ for every $\sx$ $\upsilon$-a.e.,
because $\overline{\co}(L)\subset \co(\overline L)$.
Set $\varepsilon>0$ and $\delta>0$. By the tightness of the r.v. $(p_{t_n}\circ \aver,n\in \bN)$, choose a compact set
$K$ such that $\upsilon(p_{t_n}\circ \aver)^{-1}(K^c)\leq \delta$ for all $n$. As $\overline{L_\varepsilon}^c$ is an open set,
one has
$$
\upsilon g^{-1}(\overline{L_\varepsilon}^c)\leq \lim_n \upsilon (p_{t_n}\circ\aver)^{-1}(\overline{L_\varepsilon}^c)\leq \lim_n \upsilon (p_{t_n}\circ\aver)^{-1}(\overline{L_\varepsilon}^c\cap K) + \delta\,.
$$
Let $\sx\in \Phi(E)$ be fixed. By contradiction, suppose that $\1_{\overline{L_\varepsilon}^c\cap K}(p_{t_n}(\aver(\sx)))$ does 
not converge to zero. Then, $p_{t_n}(\aver(\sx))\in \overline{L_\varepsilon}^c\cap K$ for every $n$ along some subsequence.
As $K$ is compact, one extract a subsequence, still denoted by $t_{n}$, s.t. $p_{t_{n}}(\aver(\sx))$ converges.
The corresponding limit must belong to the closed set $L_\varepsilon^c$, 
but must also belong to $L$ by definition of $\sx$.
This proves that $\1_{L_\varepsilon^c\cap K}(p_{t_n}\circ \aver(\sx)))$ converges to zero for all $x\in \Phi(E)$.
As $\support(\upsilon)\subset \Phi(E)$, 
$\1_{\overline{L_\varepsilon}^c\cap K}(p_{t_n}\circ \aver)$ converges to zero $\upsilon$-a.s.
By the dominated convergence theorem, we obtain that $\upsilon g^{-1}(\overline{L_\varepsilon}^c)\leq \delta$. Letting $\delta\to 0$
we obtain that $\upsilon g^{-1}(\overline{L_\varepsilon}^c)=0$.
Hence, $g(\sx)\in \overline L$ for all $\sx$ $\upsilon$-a.e. The proof is complete.

\subsection{Proof of Theorem~\ref{cvg-XY}}

Recall the definition $\mU \eqdef \bigcup_{\pi \in \cI(\Phi)} \support(\pi)$. 
By Th.~\ref{the:CV}, for all $\varepsilon > 0$, 
\[
\limsup_{n\to\infty} \bE^{\nu,\gamma} \Lambda_n( \mU_\varepsilon^c ) \xrightarrow[\gamma\to 0]{} 0 ,
\]
where $\Lambda_n$ is the random measure given by~(\ref{eq:Lambdan}).
By Theorem~\ref{poincare}, $\support(\pi) \subset \text{BC}_\Phi$ for each 
$\pi \in \cI(\Phi)$. Thus, $\mU_\varepsilon\subset(\text{BC}_\Phi)_\varepsilon$.
Hence, $\limsup_n \bE^{\nu,\gamma}\Lambda_n( ( (\text{BC}_\Phi)_\varepsilon)^c ) 
\to 0$ as $\gamma\to 0$. This completes the proof.

\rev{
\section{Applications}
\label{sec:applis}
In this section, we return to the Examples \ref{ex:optim} and \ref{ex:fluid} of Section~\ref{sec:examples}.

\subsection{Non-Convex Optimization}

Consider the algorithm~\eqref{eq:prox-gradient} to solve
problem~\eqref{eq:pb-nonCVX} where $\ell : \Xi \times E \to \bR$, $r :
E \to \bR$ and $\xi$ is a random variable over a probability space
$(\Omega,\mcF, \bP)$ with values in the measurable space $(\Xi,
{\mcG})$ and with distribution $\mu$. Assume that $\ell(\xi,\,.\,)$ is
continuously differentiable for every $\xi \in \Xi$, that
$\ell(\,.\,,x)$ is $\mu$-integrable for every $x \in E$ and that $r$
is a convex and lower semicontinuous function. We assume that for every compact subset $K$ of $E$,
there exists $\epsilon_K > 0$ s.t.
\begin{equation}
  \label{eq:moment-l}
  \sup_{x \in K} \int \|\nabla \ell(s,x)\|^{1+\epsilon_K} \mu(ds) < \infty\,.
\end{equation}
Define $L(x) \eqdef \bE_\xi(\ell(\xi,x))$. Under Condition (\ref{eq:moment-l}),
it is easy to check that $L$ is differentiable, and that $\nabla L(x) = \int \nabla \ell(s,x) \mu(ds)$.
From now on, we assume moreover that $\nabla L$ is Lipschitz continuous. 
Letting $H(s,x)\eqdef -\nabla \ell(s,x) -\partial r(x)$, it holds that $H(\,.\,,x)$ is proper, $\mu$-integrable and usc~\cite{phe-97},
and that the corresponding selection integral $\sH(x) \eqdef \int H(s,x)\mu(ds)$ is given by
$$
\sH(x) = -\nabla L(x) - \partial r(x)\,.
$$
By \cite[Theorem 3.17, Remark 3.14]{bre-livre73},  for every $a\in E$, 
the DI $\dot \sx(t)\in \sH(x(t))$ admits a unique solution on $[0,+\infty)$
s.t. $\sx(0)=a$. 

Now consider the iterates $x_n$ given
by~\eqref{eq:prox-gradient}. They satisfy~(\ref{eq:iterative-model})
where
$h_\gamma(s,x) \eqdef \gamma^{-1}(\prox_{\gamma r}(x-\gamma \nabla
\ell(s,x)) - x)$.
We verify that the map $h_\gamma$ satisfies Assumption (RM).  Let us
first recall some known facts about proximity operators.  Using
\cite[Prop. 12.29]{bau-com-livre11}, the mapping
$x\mapsto \gamma^{-1}(x-\prox_{\gamma r}(x))$ coincides with the
gradient $\nabla r_\gamma$ of the Moreau enveloppe
$r_\gamma : x\mapsto \min_y r(y) + \|y-x\|^2$.  By
\cite[Prop. 23.2]{bau-com-livre11}, $\nabla r_\gamma(x)\in \partial r(\prox_{\gamma r}(x))$, 
for every $x\in E$.  Therefore,
\begin{align}
  h_{\gamma}(s,x) 
 &= -\nabla r_\gamma(x-\gamma \nabla \ell(s,x)) -\nabla\ell(s,x) \label{eq:tmp-b}  \\
 &\in -\partial r(\prox_{\gamma r}(x-\gamma \nabla \ell(s,x))) -\nabla\ell(s,x) \nonumber \\
 &\in -\partial r(x -\gamma h_\gamma(s,x)) -\nabla\ell(s,x) \,. \label{eq:tmp-a}
\end{align}
In order to show that Assumption (RM)-\ref{hyp:RM-cvg}) is satisfied, we need some estimate on $\|h_\gamma(s,x)\|$.
Using Eq.~(\ref{eq:tmp-b}) and the fact that $\nabla r_\gamma$ is $\gamma^{-1}$-Lipschitz continuous
(see \cite[Prop. 12.29]{bau-com-livre11}), we obtain that
\begin{align}
  \|h_\gamma(s,x)\|&\leq \| \nabla r_\gamma(x)\| + 2\|\nabla \ell(s,x)\| \nonumber \\
&\leq \|\partial^0 r(x)\|+2\|\nabla \ell(s,x)\| \,,\label{eq:bound-hgamma}
\end{align}
where  $\partial^0 r(x)$ the least norm element in $\partial r(x)$ for every $x \in E$,
and where the last inequality is due to \cite[Prop. 23.43]{bau-com-livre11}.
As  $\partial^0 r$ is locally bounded  and $\partial r$ is usc, 
it follows from Eq. (\ref{eq:tmp-a}) that Assumption (RM)-\ref{hyp:RM-cvg}) is satisfied.
The estimate (\ref{eq:bound-hgamma}) also yields Assumption (RM)-\ref{hyp:RM-moments}).
As a conclusion, Assumption (RM) is satisfied.
In particular, the statement of Th.~\ref{th:SA=wAPT} holds.
\medskip

To show that Assumption (PH) is satisfied, we first recall the Proximal
Polyak-Lojasiewicz (PPL) condition introduced
in~\cite{karimi2016linear}. Assume that $L$ is differentiable with a
$C$-Lipschitz continuous gradient. We say that $L$ and $r$ satisfy the
(PPL) condition with constant $\beta > 0$ if for every $x \in
E$, $$\frac12 D_{L,r}(x,C) \geq \beta \left[(L+r)(x) - \min
  (L+r)\right]$$ where $$D_{L,r}(x,C) \eqdef -2 C \min_{y \in E}
\left[\ps{\nabla L(x),y-x} + \frac{C}2 \|y - x\|^2 + r(y) -
  r(x)\right].$$ The (PPL) helps to prove the convergence of the
(deterministic) proximal gradient algorithm applied to the
(deterministic) problem of minimizing the sum $L+r$. We refer to~\cite{karimi2016linear} for practical
cases where the (PPL) condition is satisfied. In our stochastic setting, we introduce the Stochastic PPL condition (SPPL). We say that $\ell$ and $r$ satisfy the (SPPL) condition if there exists $\beta > 0$ such that for every $x \in E$,
$$\frac12 \int D_{\ell(s,\cdot),r}(x,\frac{1}{\gamma}) \mu(ds) \geq \beta \left[(L+r)(x) - \min (L+r)\right].
$$
for all $\gamma \leq \frac{1}{C}$. Note that (SPPL) is satisfied if for every $s \in \Xi$, $\ell(s,\cdot)$ and $r$ satisfy the (PPL) condition with constant $\beta$. In the sequel, we assume that for every $x \in E$, the random variable $\|\ell(x,\xi)\|$ is square integrable and denote by $W(x)$ its variance.

\begin{proposition}
\label{th:PHnoncvx}
Assume that the (SPPL) condition is satisfied, that $\gamma \leq \frac{1}{C}$ and that $$\beta (L(x)+r(x)) - W(x) - \frac{1}4 \|\nabla L(x)\|^2 \longrightarrow_{\|x\| \to +\infty} +\infty.$$ Then (PH) is satisfied.
\end{proposition}
\begin{proof}
Using (sub)differential calculus, it is easy to show that for every $n \in \bN$, \begin{equation*}
x + \gamma h_\gamma(s,x) = \arg\min_{y \in E} \left[\ps{\nabla \ell(s,x),y-x} + \frac{1}{2\gamma} \|y - x\|^2 + r(y) - r(x)\right].
\end{equation*}
Since $\nabla L$ is $1/\gamma$-Lipschitz continuous, 
\begin{align}
\label{eq:PHnoncvx}
(L+r)(x + \gamma h_\gamma(s,x)) &= L(x + \gamma h_\gamma(s,x)) + r(x) + r(x + \gamma h_\gamma(s,x)) - r(x) \nonumber\\
&\leq (L+r)(x) + \ps{\nabla L(x), \gamma h_\gamma(s,x)} + \frac{1}{2\gamma}\|\gamma h_\gamma(s,x)\|^2 \nonumber\\
&\phantom{=} + r(x + \gamma h_\gamma(s,x)) - r(x) \nonumber\\
&\leq (L+r)(x) + \ps{\nabla \ell(s,x), \gamma h_\gamma(s,x)} + \frac{1}{2\gamma}\|\gamma h_\gamma(s,x)\|^2 \nonumber\\
&\phantom{=} + \ps{\nabla L(x) - \nabla \ell(s,x), \gamma h_\gamma(s,x)} + r(x + \gamma h_\gamma(s,x)) - r(x) \nonumber\\
&\leq (L+r)(x) - \frac{\gamma}{2} D_{\ell(s,\cdot),r}(x,1/\gamma) \nonumber\\
&\phantom{=} + \gamma\ps{\nabla \ell(s,x) - \nabla L(x), \nabla \ell(s,x) + \nabla r_\gamma(x - \gamma \nabla \ell(s,x))}
\end{align}
Recall that for every $x,y \in E$,
\begin{align*}
\ps{\nabla r_\gamma(x) - \nabla r_\gamma(y),x-y} & =
\ps{\nabla r_\gamma(x) - \nabla r_\gamma(y),\prox_{\gamma r}(x)-\prox_{\gamma r}(y)} \\
&\phantom{=} + \ps{\nabla r_\gamma(x) - \nabla r_\gamma(y),\gamma \nabla r_\gamma(x) - \gamma \nabla r_\gamma(y)}\\
& \geq \gamma \|\nabla r_\gamma(x) - \nabla r_\gamma(y)\|^2, 
\end{align*}
using the monotonicity of $\partial r$. Hence,
\begin{equation*}
\ps{\nabla r_\gamma(x - \gamma \nabla \ell(s,x)) - \nabla r_\gamma(x),\gamma \nabla \ell(s,x)}
 \leq -\gamma \|\nabla r_\gamma(x) - \nabla r_\gamma(x - \gamma \nabla \ell(s,x))\|^2. 
\end{equation*}
Therefore, 
\begin{align*}
&\gamma\ps{\nabla \ell(s,x) - \nabla L(x), \nabla r_\gamma(x - \gamma \nabla \ell(s,x)) - \nabla r_\gamma(x)} \\
\leq & -\gamma \|\nabla r_\gamma(x) - \nabla r_\gamma(x - \gamma \nabla \ell(s,x))\|^2 \\
& + \gamma \|\nabla r_\gamma(x) - \nabla r_\gamma(x - \gamma \nabla \ell(s,x))\|^2 + \frac{\gamma}{4}\|\nabla L(x)\|^2 \\
\leq & \frac{\gamma}{4}\|\nabla L(x)\|^2,
\end{align*}
where we used $\ps{x,y} \leq \|x\|^2 + \frac{1}4 \|y\|^2$.

Plugging into~\eqref{eq:PHnoncvx},

\begin{align*}
(L+r)(x + \gamma h_\gamma(s,x)) &\leq (L+r)(x) - \frac{\gamma}{2} D_{\ell(s,\cdot),r}(x,1/\gamma) \nonumber\\
&\phantom{=} + \gamma\ps{\nabla \ell(s,x) - \nabla L(x), \nabla \ell(s,x)} \nonumber\\
&\phantom{=} + \gamma\ps{\nabla \ell(s,x) - \nabla L(x), \nabla r_\gamma(x - \gamma \nabla \ell(s,x)) - \nabla r_\gamma(x)} \nonumber
\\
&\phantom{=} + \gamma\ps{\nabla \ell(s,x) - \nabla L(x), \nabla r_\gamma(x)} \nonumber\\
&\leq (L+r)(x) - \frac{\gamma}{2} D_{\ell(s,\cdot),r}(x,1/\gamma) \nonumber\\
&\phantom{=} + \gamma\ps{\nabla \ell(s,x) - \nabla L(x), \nabla \ell(s,x)} \nonumber\\
&\phantom{=} + \frac{\gamma}{4}\|\nabla L(x)\|^2 \nonumber
\\
&\phantom{=} + \gamma\ps{\nabla \ell(s,x) - \nabla L(x), \nabla r_\gamma(x)} \nonumber
\end{align*}

Integrating with respect to $\mu$, we obtain 
\begin{align*}
\int (L+r)(x + \gamma h_\gamma(s,x)) \mu(ds)&\leq (L+r)(x) - \gamma \beta \left((L+r)(x) - \min(L+r)\right)\\
&\phantom{=} +\gamma W(x) + \frac{\gamma}{4} \|\nabla L(x)\|^2.
\end{align*}
Finally, the condition (PH) is satisfied with $\alpha(\gamma) = \gamma$, $\beta(\gamma) = 0$, $V = L+r - \min L+r$ and 
$$\psi = \beta V - W - \frac{1}4 \|\nabla L\|^2.$$

\end{proof}
Note that the assumptions of Proposition~\ref{th:PHnoncvx} are satisfied if the (SPPL) condition is satisfied, $L(x) + r(x) \rightarrow_{\|x\| \to +\infty} +\infty$ and the function $x \mapsto \int \|\nabla \ell(s,x)\|^2 \mu(ds)$ is bounded.
\medskip

The condition (FL) is naturally satisfied. Identifying the invariant measures of the DI, we finally obtain a long-run convergence result for the algorithm~\eqref{eq:prox-gradient}. Let $\nu\in\cM(E)$ s.t. $\nu(L+r)<\infty$. Let $\mZ = \{x \in E, \text{ s.t }0 \in \nabla L(x) + \partial r(x)\}$. For all 
$\varepsilon > 0$, 
\begin{equation}
\limsup_{n\to\infty} \frac 1{n+1}\sum_{k=0}^n \bP^{\nu,\gamma}( d(X_k,\mZ)>\varepsilon)\xrightarrow[\gamma\to 0]{}0\,.
\end{equation}

\subsection{Fluid Limit of a System of Parallel Queues}

We now apply the results of this paper to the dynamical system described in
Example~\ref{ex:fluid} above. For a given $\gamma > 0$, the transition 
kernel $P_\gamma$ of the Markov chain $(x_n)$ whose entries are given by 
Eq.~\eqref{eq:queue} is defined on $\gamma\bN^N \times 2^{\gamma \bN^N}$.  
This requires some small adaptations of the statements of the main results
that we keep confined to this paragraph for the paper readability. 
The limit behavior of the interpolated process (see Theorem~\ref{th:SA=wAPT}) 
is described by the following proposition, which has an analogue 
in~\cite{gas-gau-12}: 
\begin{proposition}
For every compact set $K\subset \bR^N$, the family 
$\{ \bP^{a,\gamma}\sX_\gamma^{-1},a\in K\cap\gamma\bN^N ,0<\gamma<\gamma_0 \}$ 
is tight.  Moreover, for every $\varepsilon>0$, 
\[
\sup_{a\in K \cap \gamma\bN^N} 
\,\bP^{a,\gamma}\left(d(\sX_\gamma,\Phi_\sH(K))>\varepsilon\right)
 \xrightarrow[\gamma\to 0]{}0\, , 
\]
where the set-valued map $\sH$ is given by \eqref{Hqueue}.  
\end{proposition} 
\begin{proof} 
To prove this proposition, we mainly need to check that Assumption (RM) is
verified.  We recall that the Markov chain $(x_n)$ given by
Eq.~\eqref{eq:queue} admits the representation~\eqref{eq:decomp-markov-drift},
where the function $g_\gamma = (g_\gamma^1,\ldots,g_\gamma^N)$ is given 
by~\eqref{gk-queue}. If we set $h_\gamma(s,x) = g_\gamma(x)$ (the fact that
$g_\gamma$ is defined on $\gamma\bN^N$ instead of $\bR_+^N$ is irrelevant),  
then for each sequence $(u_n, \gamma_n) \to (u^\star, 0)$ with 
$u_n \in \gamma_n \bN^N$ and $x^\star \in \bR_+^N$, it holds that 
$g_{\gamma_n}(u_n) \to \sH(u^\star)$.  
Thus, Assumption (RM)--\ref{hyp:drift}) is verified with $H(s,x) = \sH(x)$. 
Assumptions (RM)--\ref{hyp:RM-cvg}) to (RM)--\ref{hyp:RM-integ}) are obviously 
verified. Since the set-valued map $\sH$ satisfies the 
condition~\eqref{eq:lin-growth}, Assumption (RM)--\ref{hyp:flot-borne}) is 
verified. Finally, the finiteness assumption~\eqref{eq:moment} with 
$\epsilon_K = 2$ follows from the existence of second moments for the $A^k_n$, 
and \eqref{eq:moment-bis} is immediate. The rest of the proof follows word for
word the proof of Theorem~\ref{th:SA=wAPT}. 
\end{proof}

The long run behavior of the iterates is provided by the following proposition:
\begin{proposition}
Let $\nu\in\cM(\bR_+^N)$ be such that $\nu(\|\cdot\|^2)<\infty$. For each 
$\gamma > 0$, define the probability measure $\nu_\gamma$ on $\gamma \bN^N$ as 
\[
\nu_\gamma(\{\gamma i_1, \gamma i_2, \ldots, \gamma i_N\}) = 
\nu(\gamma (i_1 - 1/2, i_1+1/2] \times \cdots \times 
  \gamma (i_N - 1/2, i_N + 1/2]) \, . 
\]
If Condition~\eqref{eq:stability-queue} is satisfied, then for all 
$\varepsilon > 0$,
$$
\limsup_{n\to\infty}\ \frac 1{n+1}\sum_{k=0}^n 
 \bP^{\nu_\gamma,\gamma}\left(d\left(X_k\,, 0 \right) 
   \geq \varepsilon\right) \xrightarrow[\gamma\to 0]{}0\,.
$$
\end{proposition} 

To prove this proposition, we essentially show that the assumptions of 
Theorem~\ref{cvg-XY} are satisfied. In the course of the proof, we shall
establish the existence of the (PH) criterion with a function $\psi$ having 
a linear growth. With some more work, it is possible to obtain a (PH)
criterion with a faster than linear growth for $\psi$, allowing to 
obtain the ergodic convergence as shown in Theorem~\ref{cvg-CVSI}. 
This point will not be detailed here. 

\begin{proof} 
Considering the space $\gamma\bN^N$ as a metric space equipped with the 
discrete topology, any probability transition kernel on 
$\gamma\bN^N \times 2^{\gamma \bN^N}$ is trivially Feller. Thus, 
Proposition~\ref{prop:feller} holds when letting $P = P_\gamma$  
and $\nu \in \cM(\gamma \bN^N)$. Let us check that Assumption (PH) is verified 
if the stability condition~\eqref{eq:stability-queue} is satisfied. Let 
\[
V : \bR_+^N \to \bR_+, \quad 
x = (x^1,\ldots, x^N) \mapsto \Bigl(\sum_{k=1}^N x^k/\eta^k \Bigr)^2 \, . 
\]
Given $1\leq k,\ell \leq N$, define $f(x) = x^k x^\ell$ on $\gamma\bN^2$. 
Using Eq.~\eqref{eq:queue}, the iid property of the process $((A^1_n,\ldots, 
A^N_n, B^1_n,\ldots, B^N_n), n\in\bN)$ and the finiteness of the second 
moments of the $A^k_n$, we obtain 
\begin{align*} 
(P_\gamma f)(x) &\leq x^k x^\ell 
  - \gamma x^k \left( 
  \eta^\ell \1_{\{x^{\ell}>0,\,x^{\ell-1} =\cdots = x^{1} = 0\}} 
  - \lambda^\ell \right)  \\
&\phantom{=} - \gamma x^\ell \left( 
  \eta^k \1_{\{x^{k}>0,\,x^{k-1} =\cdots = x^{1} = 0\}} - \lambda^k \right) 
    + \gamma^2 C \, , 
\end{align*} 
where $C$ is a positive constant. Thus, when $x \in \gamma\bN^N$, 
\begin{align*}
(P_\gamma V)(x) &\leq V(x) - 
  2 \gamma \sum_{k=1}^N x^k / \eta^k \sum_{\ell=1}^N 
  \left( \1_{\{x^{\ell}>0,\,x^{\ell-1} =\cdots = x^{1} = 0\}} - 
  \lambda^\ell / \eta^\ell \right) + \gamma^2 C , 
\end{align*} 
after modifying the constant $C$ if necessary. If $x\neq 0$, then one and only
one of the $\1_{\{x^{\ell}>0,\,x^{\ell-1} =\cdots = x^{1} = 0\}}$ is equal to 
one. Therefore, 
$(P_\gamma V)(x) \leq V(x) - \gamma \psi(x) + \gamma^2 C$, 
where 
\[ 
  \psi(x) = 2 \Bigl( 1 - \sum_{\ell=1}^N \lambda^\ell / \eta^\ell \Bigr) 
  \sum_{k=1}^N x^k / \eta^k \, . 
\]
As a consequence, when Condition~\eqref{eq:stability-queue} is satisfied, the
function $\psi$ is coercive, and one can straightforwardly check that the 
statements of Proposition~\ref{prop:tight}--\ref{it:itight}) and 
Proposition~\ref{prop:tight}--\ref{it:mtight}) hold true under minor 
modifications, namely,  
$\bigcup_{P \in \cP}\cI(P)$ is tight in $\cM(\bR_+^N)$, since  
$\sup_{\pi\in\cI(\cP)} \pi(\psi) < +\infty$, where 
$\cP = \{ P_\gamma \}_{\gamma\in(0,\gamma_0)}$. Moreover, for every 
$\nu\in\cM(\bR_+^N)$ s.t. $\nu(\|\cdot\|^2)<\infty$ and every $P\in \cP$, 
$\{\bE^{\nu_\gamma,P_\gamma}\Lambda_n\,,\,\gamma\in (0,\gamma_0), n\in\bN \}$ 
is tight, since 
$\sup_{\gamma\in(0,\gamma_0),n\in\bN}
 \bE^{\nu_\gamma,P_\gamma}\Lambda_n(\psi)<\infty$. 
We can now follow the proof of Theorem~\ref{cvg-XY}. Doing so, all it remains 
to show is that the Birkhoff center of the flow $\Phi_\sH$ is reduced to 
$\{ 0 \}$. This follows from the fact that when 
Condition~\eqref{eq:stability-queue} is satisfied, all the trajectories of
the flow $\Phi_\sH$ converge to zero, as shown in \cite[\S~3.2]{gas-gau-12}. 
\end{proof}  

}  


\def\cprime{$'$} \def\cdprime{$''$} \def\cprime{$'$}

\end{document}